\documentclass[preprint, 5p, twocolumn, 12pt]{elsarticle}
\usepackage{amsthm}
\usepackage{amssymb}
\usepackage[active]{srcltx}
\usepackage{amsfonts}
\usepackage{graphics}
\usepackage{eurosym}
\usepackage{pstricks}
\usepackage{shortvrb}
\usepackage{graphicx}
\usepackage{subfigure}
\usepackage{color}
\usepackage{multirow}
\usepackage{float}
\usepackage{ifpdf}
\usepackage{multicol}
\usepackage{epstopdf}
\usepackage{lineno}
\biboptions{sort&compress}
\usepackage{array}
\usepackage{supertabular}
\usepackage{xfrac}
\usepackage{amsmath}
\usepackage{amssymb}
\usepackage{framed}
\usepackage{enumitem}
\usepackage[boxruled,vlined]{algorithm2e}
\SetKwRepeat{Do}{do}{while}
\usepackage{tikz}
\usetikzlibrary{shapes,arrows}

\makeatletter
\newtheoremstyle{break}
   {\topsep}{\topsep}
   {\itshape}{}
   {\bfseries}{}
   {\newline}
   {\thmname{#1}\thmnumber{\@ifnotempty{#1}{ }\@upn{#2}}%
    \thmnote{ {\bfseries(#3)}}}
\makeatother
\theoremstyle{break}

\newtheorem{theorem}{Theorem}

\newtheorem{definition}{Definition}

\usepackage{url}
\usepackage{color}
\usepackage{soul}
\soulregister\cite7
\soulregister\ref7
\soulregister\eqref7
\usepackage{alphalph}
\usepackage{etoolbox}
\usepackage{tabularx}

\patchcmd{\subequations}{\alph{equation}}{\alphalph{\value{equation}}}{}{}
\journal{Electrical Power Systems Research}
\begin{document}
\begin{frontmatter}

\title{Properties of Convex Optimal Power Flow Model {Based on Power Loss Relaxation} }
\author{Zhao Yuan\corref{cor}}
\cortext[cor]{Corresponding author}
\ead{zhao.yuan@epfl.ch}
\author{Mario~Paolone}
\address{Distributed Electrical Systems Laboratory\\ \'Ecole polytechnique f\'ed\'erale de Lausanne (EPFL), Lausanne, 1015, Switzerland.}

\begin{abstract}
We derive the branch ampacity constraint associated to power losses for the convex optimal power flow (OPF) model based on the branch flow formulation. The branch ampacity constraint derivation is motivated by the physical interpretation of the transmission line $\Pi$-model and practical engineering considerations. We rigorously prove {and} derive: (i) the {loop} constraint of voltage phase angle, required to make the branch flow model valid for meshed power networks, is a relaxation of the original nonconvex alternating current optimal power flow (o-ACOPF) model; (ii) {the} necessary conditions to recover a feasible solution of the o-ACOPF model from the optimal solution of the convex second-order cone ACOPF (SOC-ACOPF) model; (iii) the expression of the global optimal solution of the o-ACOPF model providing that the relaxation of the SOC-ACOPF model is tight; (iv) the (parametric) optimal value function of the o-ACOPF or SOC-ACOPF model is monotonic with regarding to the power loads if the objective function is monotonic with regarding to the nodal power generations; (v) tight solutions of the SOC-ACOPF model always exist when the power {loads are} sufficiently large. Numerical experiments {using benchmark} power networks to validate our {findings and to compare with other convex OPF models,} are given and discussed.
\end{abstract}
\end{frontmatter}

\section*{Nomenclature}
\setlength{\parindent}{0.5em}
\textbf{Sets:}\\
\begin{tabular}{l l}
$\mathcal{N}$ & Nodes (or buses). \\
$\mathcal{L}$ & Lines (or branches). \\
$\mathcal{C}$ & Cycles (or closed loops).\\
\end{tabular}

\textbf{Parameters:} \\
\begin{tabular}{l l}
$A^{+}_{nl},A^{-}_{nl}$ & Node to branch incidence matrices.\\
$X_{l},R_{l}$ & Longitudinal reactance and resistance\\
&of branch $l$ modelled as a $\Pi$-model.\\
$G_{n},B_{n}$ & Shunt conductance and susceptance\\
&of node $n$.\\
$B_{s_{l}},B_{r_{l}}$ & Sending- and receiving- end shunt\\
&susceptance of branch $l$.\\
\end{tabular}
\begin{tabular}{l l}
$\widetilde{K}_l, K_l$ & Actual and approximated ampacity\\
&of branch $l$.\\
$p^{min}_{n}, p^{max}_{n}$ & Lower and upper bounds of $p_{n}$.\\
$q^{min}_{n}, q^{max}_{n}$ & Lower and upper bounds of $q_{n}$.\\
$\theta^{min}_{l}, \theta^{max}_{l}$ & Lower and upper bound of $\theta_{l}$.\\
$p_{d_{n}},q_{d_{n}}$ & Active and reactive power loads \\
&of node $n$. \\
$\alpha_{n}, \beta_{n}, \gamma_{n}$ & Cost parameters of active power\\
&generation.\\
\end{tabular}
\textbf{Variables:} \\
\begin{tabular}{l l}
$p_{n},q_{n}$ & Active and reactive power injections \\
&at node $n$. \\
$p_{s_{l}},q_{s_{l}}$ & Non-measurable sending-end active\\
&and reactive power flows of branch $l$. \\
$\widetilde{p}_{s_{l}},\widetilde{q}_{s_{l}}$ & Measurable sending-end active\\
&and reactive power flows of branch $l$. \\
\end{tabular}
\begin{tabular}{l l}
$p_{r_{l}},q_{r_{l}}$ & Non-measurable receiving-end active\\
&and reactive power flows of branch $l$. \\
$\widetilde{p}_{r_{l}},\widetilde{q}_{r_{l}}$ & Measurable receiving-end active and \\
&reactive power flows of branch $l$. \\
$q_{cs_{l}}, q_{cr_{l}}$ & Sending- and receiving- end shunt \\
&reactive power of branch $l$.\\
$i_{s_{l}},i_{r_{l}}$ & Non-measurable sending- and \\
&receiving- end current of branch $l$.\\
$\widetilde{i}_{s_{l}},\widetilde{i}_{r_{l}}$ & Measurable sending- and receiving-\\
&end current of branch $l$.\\
$p_{o_{l}}, q_{o_{l}}$ & Active and reactive power losses of \\
&branch $l$.\\
$v_{n},V_{n}$ & Phase-to-ground voltage magnitude \\
&and voltage square at node $n$.\\
$v_{s_{l}}, v_{r_{l}}$ & Sending- and receiving- end phase-to\\ 
&-ground voltage of branch $l$.\\
$V_{s_{l}}, V_{r_{l}}$ & Sending- and receiving- end phase-to\\ 
&-ground voltage square of branch $l$.\\
$\theta_{n}$& Phase-to-ground voltage phase angle \\
&at node $n$.\\
$\theta_{l}$& Phase angle difference between \\
&the sending- and receiving- end \\
&phase-to-ground voltages of branch $l$.\\
$\theta_{s_{l}}, \theta_{r_{l}}$& Phase angles of sending- and \\ 
&receiving- end phase-to-ground\\
&voltages of branch $l$.\\
$K_{o_{_l}}$ & Equivalent ampacity constraint for \\ 
&power losses of branch $l$.\\
\end{tabular}

\textbf{Indicators:} \\
\begin{tabular}{l l}
$Gap^{po}_{l},Gap^{qo}_{l}$ &Relaxation gaps of active\\ 
&power loss and reactive \\
&power loss.\\
$Gap^{po, max},Gap^{qo, max}$ &Maximum relaxation gaps\\
&of $Gap^{po}_{l},,Gap^{qo}_{l}$.\\
\end{tabular}
\section{Introduction}
Optimal power flow (OPF) is a fundamental mathematical optimization model for decision making in power system operation and planning \cite{carpentier1962contribution}. Improving the solution quality of OPF can give large economic and engineering benefits to the power industry \cite{Cain2012, advan_opt_meth_ps}. Recent literature focusing on the convexification of the OPF model suffered from {the} inexactness due to the relaxation of several constraints \cite{Farivar2012_full,Gan2015,Kocuk2015,Bai2008,Lavaei2012,moment_relax_opf,sparcity_moment,qc_relax_minlp,Coffrin2015,stren_qc_bdt_pno, ttn_mcrelax_to_glob}.

The branch flow formulation of the power flow equations for radial power networks has been originally derived by Baran and Wu in \cite{Capa_place_barawu} to optimize the placement of capacitor in radial distribution networks. In \cite{Jabr2006}, Jabr derives a conic programming approach to solve load flow in radial distribution network. In \cite{Farivar2012_full}, Farivar and Low propose the branch flow model as a relaxed OPF model by relying on second-order cone programming (SOCP) and prove that the relaxation is tight for radial networks under the assumption that the upper bound of power generation is infinite. In \cite{Gan2015}, Gan, Li et al. prove that the optimal solution of the branch flow model is exact (tight) if the voltage upper bounds are not binding and the network parameters satisfy some mild conditions which can be checked \textit{a priori}. Christakou, Tomozei et al. show in \cite{Christakou2015} that the branch flow model is not exact due to the approximation of the {branch} ampacity constraint. Nick et al. propose an exact convex OPF model for radial distribution networks in \cite{Exact_OPF_rad} by considering the shunt parameters associated to the exact modelling of the lines or other branch elements. Sufficient conditions regarding the network parameters under which the convex OPF model in \cite{Exact_OPF_rad} can give exact solutions to the original ACOPF are provided and rigorously proved. In \cite{Kocuk2015}, Kocuk, Dey et al.  propose three methods (arctangent envelopes, dynamic linear inequalities and separation over cycle constraints) to strengthen or tighten {the} SOCP relaxation of OPF. Numerical results show better solution quality of SOCP-based model over SDP-based model \cite{Kocuk2015}.

Semidefinite programming (SDP) relaxation of OPF has been firstly proposed by Bai et al. in \cite{Bai2008}. The proposed procedures derive the rectangular form OPF in a quadratic programming model and replace the variable-vector ($\textbf{x}$) by the variable-matrix ($X=\textbf{x}^{T}\textbf{x}$). Solutions of OPF can be recovered from the square roots of the diagonal elements in the variable-matrix \cite{Bai2008}. SDP advantages in avoiding the derivation of Jacobian and Hessian matrices if the interior point method (IPM) is used \cite{Bai2008}. Lavaei and Low propose to solve the SDP relaxation of the dual OPF problem \cite{Lavaei2012}. They prove that sufficient and necessary condition of zero duality gap holds for several IEEE test cases (14, 30, 57, 118) at {the} base power load {level}.\footnote{As a comparison, in this paper, we evaluate the OPF solutions of test cases at low power load levels (for which the relaxation gaps are prominent).} But the branch ampacity constraint is not fully addressed. Especially, as it is proved in \cite{Christakou2015, Exact_OPF_rad}, the physical interpretation of the transmission line $\Pi$ model is not correct in \cite{Lavaei2012}. In \cite{Convex_Rel_mesh}, Madani, Sojoudi et al. show that SDP relaxation works well if the line capacity is expressed by nodal voltage variables. A penalization method is proposed in \cite{Convex_Rel_mesh} to reduce the rank of the SDP solution. Recenly, Eltved, Dahl et al. numerically investigate the computational efficiency of SDP by using test cases up to 82,000 nodes which is solved in around 7 hours using a high performance computing node \cite{SDP_rob_sca}.

Based on polynomial optimization theory, Molzahn and Hiskens propose the moment-based relaxations of OPF in \cite{moment_relax_opf}. This formulation firstly defines order-$\gamma$ moment relaxation $x_{\gamma}$ of all monomials $\hat{x}^{\alpha}$ of voltage real-and-imaginary components $\hat{x}$. Then, all the monomials $\hat{x}^{\alpha}$ up to order $2\gamma$ constitute the symmetric moment matrix $M_{\gamma}$ which is used to re-formulate the OPF constraints via the SDP. Global optimal solutions are found, at the cost of heavy computational burden due to higher relaxation order $\gamma$, for the test cases in \cite{moment_relax_opf}. The same authors in \cite{sparcity_moment} improve the computational efficiency of this method by exploiting power system sparsity and applying high relaxation order to specific buses.
In \cite{qc_relax_minlp}, Hijazi, Coffrin et al. propose a quadratic convex (QC) relaxation by replacing the nonconvex voltage-amplitude-and-phase-angle associated constraints with the corresponding convex envelopes. In \cite{Coffrin2015}, the same authors investigate the relationships between different convex OPF formulations including QC, SDP and SOCP. Reducing the optimality gap of QC or SOCP based OPF models by bound tightening techniques can be found in \cite{stren_qc_bdt_pno, ttn_mcrelax_to_glob}.
Recently, Shchetinin et al. propose in \cite{eff_bound_tight_relax} three methods which require solving optimization problems to tighten the upper bounds of the voltage phase angle difference in order to satisfy the branch ampacity constraint. The same authors in \cite{construction_linear_app} propose an iterative algorithm to construct a number of linear constraints (based on inner or outer approximation) to approximate the branch ampacity constraint. 

In this paper, we focus on the formal derivation of equivalent ampacity constraint for the branch power losses.
More specifically, we consider to improve the convex OPF model in {radial and} meshed power networks which necessitates the additional voltage phase angle constraint. Instead of using the approach in \cite{Exact_OPF_rad} to reformulate the branch flow model, we keep using the same set of variables (in the form of power flow variables) of the branch flow model but equivalently derive the ampacity constraint for the power losses. In this way, we overcome the approximation gap due to neglecting of the shunt elements of the branches. We then propose and prove six theorems showing important properties of the proposed SOC-ACOPF model. In this regard, the main contributions of this paper are listed below. 
\begin{itemize}
\item Derivation of the transmission line or branch ampacity constraint for the power losses.
\item Proving that the SOC-ACOPF model (with additional constraint to improve its feasibility) is a relaxed ACOPF model. 
\item Deriving a feasible solution recovery procedure when the SOC relaxation is tight.
\item Demonstrating that the (parametric) optimal value functions of the SOC-ACOPF model and o-ACOPF model are monotonic with regarding to the power loads when the objective function is monotonic with regarding to the power generations.
\item Proving that larger power loads can tighten the relaxation in the SOC-ACOPF model.   
\item {Numerical Comparison of our o-ACOPF model and SOC-ACOPF model with respect to the other two convex OPF models proposed in the recent literature.}
\end{itemize}     
The rest of this paper is organized as follows. Section II formulates the o-ACOPF model and SOC-ACOPF model for {radial and} meshed power networks. Section III derives the equivalent branch ampacity constraint for power losses. Section IV proposes and proves important properties of the SOC-ACOPF model. Section V gives the numerical validations of our analytical proofs and discussions. {A numerical comparison with other convex OPF models is also conducted.} Section VI concludes.
\section{Optimal Power Flow Model}
\subsection{Power based Branch Flow Model}
We assume that the three-phase power grid satisfies two conditions: (i) all the branches and shunt impedances are circularly symmetric; (ii) all the triplets of nodal voltages and branch currents are symmetrical and balanced. These two conditions validate the use of single-line equivalent model of the three-phase power grid. The full ACOPF model is based on the validated branch flow model presented in \cite{DLMP_me}. We denote this original ACOPF model as o-ACOPF model-\eqref{NOPF}. The variable convention makes reference to the branch $\Pi$-model in Fig. \ref{fig:trans_line_pi}.  $X_{l},R_{l}$ are the longitudinal reactance and resistance of branch $l$. $B_{s_{l}},B_{r_{l}}$ are the sending- and receiving- end shunt susceptance of branch $l$. $v_{n}$ is the phase-to-ground voltage magnitude at node $n$.  $p_{s_{l}},q_{s_{l}}$ are the non-measurable sending-end active and reactive power flows {of} branch $l$. $\widetilde{p}_{s_{l}},\widetilde{q}_{s_{l}}$ are the measurable sending-end active and reactive power flows {of} branch $l$. $p_{r_{l}},q_{r_{l}}$ are the non-measurable receiving-end active and reactive power flows of branch $l$. $\widetilde{p}_{r_{l}},\widetilde{q}_{r_{l}}$ are the measurable receiving-end active and reactive power flows of branch $l$. $q_{cs_{l}}, q_{cr_{l}}$ are the sending- and receiving- end shunt reactive power of branch $l$. The subscripts $s,r,d,o,c$ in $p_{._{l}}, q_{._{l}}, v_{._{l}}, V_{._{l}}, p_{._{n}}, q_{._{n}}$ are not indices but to imply the meaning of sending-end of branch $l$, receiving-end of branch $l$, power load, power loss and shunt capacitance correspondingly. 
\begin{figure}[!htbp]
\vspace{-0.5cm}
\centering
\includegraphics[width=9.5cm]{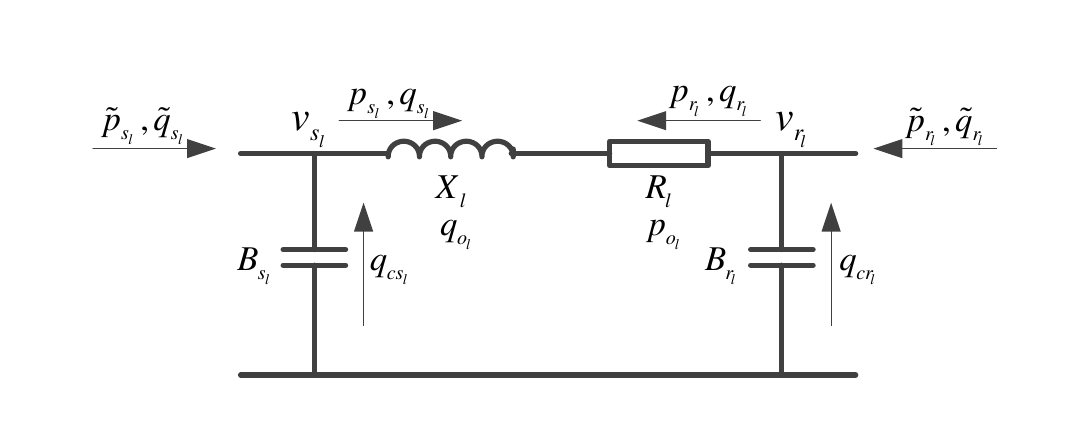}
\vspace{-0.9cm}
\caption{The branch $\Pi$-model and associated variables}\label{fig:trans_line_pi}
\end{figure}

Note $(\widetilde{p}_{s(r)_{l}},\widetilde{q}_{s(r)_{l}})$ are the actual (i.e. measurable) branch power flows. The variables $(p_{s(r)_{l}},q_{s(r)_{l}})$ are the non-measurable branch power flows used in our OPF model. For the difference between $(\widetilde{p}_{s(r)_{l}},\widetilde{q}_{s(r)_{l}})$ and $(p_{s(r)_{l}},q_{s(r)_{l}})$, or between $\widetilde{i}_{s(r)_{l}}$ and $i_{s(r)_{l}}$, please refer to references \cite{Christakou2015} and \cite{Exact_OPF_rad}. {We want to emphasize here that the physical interpretation of the branch $\Pi$-model is of importance since the branch phase-to-ground capacitance is actually distributed along the branch which means the current (or power) flowing through the branch is actually not the same along the branch due to the charging current (or power) flowing from the distributed phase-to-ground capacitance to the branch. Although the distributed branch phase-to-ground capacitance is represented in the $\Pi$-model in a lumped way at the two-ends of the branch, we should bear in mind that the current, or the power, along the branch is not the same when we consider the branch ampacity constraint. This means that the actual power flow variables $p_{s_{l}}, q_{s{l}}$ are non-measurable and difficult to be constrained. This is the main reason we only rely on the measurable power flow variables $\widetilde{p}_{s_{l}}, \widetilde{q}_{s_{l}}$ to constrain the branch ampacity and then derive the relationship between $(\widetilde{p}_{s(r)_{l}},\widetilde{q}_{s(r)_{l}})$ and $(p_{s(r)_{l}},q_{s(r)_{l}})$ later in this paper.}
\begin{subequations}\label{NOPF}
\begin{align}
    &\underset{\Omega_{ACOPF}}{\text{Minimize}} {\hspace{10pt}}  f(\Omega_{ACOPF}) \label{eq:objective_OPF}\\
    &\text{subject to} \nonumber \\
    &p_{n}-p_{d_{n}}=\sum_{l}({A^{+}_{nl}}p_{s_{l}}-{A^{-}_{nl}}p_{o_{l}})+G_{n}V_{n},\; \forall n\in \mathcal{N} \label{eq:p_balance_chap1}\\
    &q_{n}-q_{d_{n}}=\sum_{l}({A^{+}_{nl}}q_{s_{l}}-{A^{-}_{nl}}q_{o_{l}})-B_{n}V_{n},\; \forall n\in \mathcal{N}
\label{eq:q_balance_chap1}\\
 &V_{s_{l}}-V_{r_{l}}=2R_{l}p_{s_{l}}+2X_{l}q_{s_{l}}-R_{l}p_{o_{l}}-X_{l}q_{o_{l}},\; \forall l\in \mathcal{L}\label{eq:VsVr_chap2}
\end{align}
\begin{align}
&v_{s_{l}}v_{r_{l}}\sin\theta_{l}=X_{l}p_{s_{l}}-R_{l}q_{s_{l}},\, \forall l\in \mathcal{L}\label{eq:vsvrsin}\\
&\widetilde{i}^{2}_{s(r)_{l}}=\frac{\widetilde{p}^{2}_{s(r)_{l}}+\widetilde{q}^{2}_{s(r)_{l}}}{V_{s(r)_{l}}}\leq \widetilde{K}_{l},\; \forall l\in \mathcal{L}\label{eq:capacity}\\
&V_{n}=v^2_{n},\; \forall n\in \mathcal{N}\label{eq:veq}\\
&\theta_{l}=\theta_{s_{l}}-\theta_{r_{l}},\; \forall l\in \mathcal{L} \label{eq:def_thetal}\\
&{p}_{o_{l}} = \frac{ p^{2}_{s_{l}}+q^{2}_{s_{l}}}{V_{s_{l}}} R_{l},\; \forall l\in \mathcal{L}\label{eq:ploss}\\
&{q}_{o_{l}} = \frac{ p^{2}_{s_{l}}+q^{2}_{s_{l}}}{V_{s_{l}}} X_{l},\; \forall l\in \mathcal{L}\label{eq:qloss}\\
&v_{n}\in (v^{min}_{n},\, v^{max}_{n}),\; \forall n\in \mathcal{N} \label{eq:vbound}\\
&\theta_{l}\in (\theta^{min}_{l},\, \theta^{max}_{l}),\; \forall l\in \mathcal{L}\label{eq:thetal_bound}\\
&\theta_{n}\in (\theta^{min}_{n},\, \theta^{max}_{n}),\; \forall n\in \mathcal{N} \label{eq:thetan_bound}\\
&p_{n}\in (p^{min}_{n},\, p^{max}_{n}),\; \forall n\in \mathcal{N}\label{eq:pbound}\\
&q_{n}\in (q^{min}_{n}, q^{max}_{n}),\; \forall n\in \mathcal{N}\label{eq:qbound}
\end{align}

Where $\Omega_{ACOPF}=\{p_{n},q_{n},p_{s_{l}},q_{s_{l}},p_{o_{l}},q_{o_{l}},V_{n},v_{n},$ $\theta_{l},\theta_{n}\} \in \Re^{m}$ is the set of model variables. $m=5\left | \mathcal{N}  \right |+5\left | \mathcal{L}  \right |$ is the dimension of $\Omega$. $\left | \mathcal{N}  \right |$ is the cardinality of the node set $\mathcal{N}$.  $\left | \mathcal{L}  \right |$ is the cardinality of the branch set $\mathcal{L}$. $n \in \mathcal{N}$ is the index of nodes. $l \in \mathcal{L}$ is the index of branches. Based on the applications, the objective function $f(\Omega_{ACOPF})$ can be the economic cost of energy production or network power losses etc. In this paper, we assume the objective function is convex. Equations (\ref{eq:p_balance_chap1}) and (\ref{eq:q_balance_chap1}) represent the active and reactive power balance. $G_{n},B_{n}$ are the shunt conductance and susceptance of node $n$. {${A^{+}_{nl}}$ and ${A^{-}_{nl}}$} are the node-to-branch incidence matrices of the network with ${A^{+}_{nl}}=1,\, {A^{-}_{nl}}=0$ if $n$ is the sending-end of branch $l$, and ${A^{+}_{nl}}=-1,\, {A^{-}_{nl}}=-1$ if $n$ is the receiving-end of branch $l$. $p_{n},q_{n}$ are the active and reactive power generations at node $n$. $p_{d_{n}},q_{d_{n}}$ are the active and reactive power loads at node $n$. $V_{n}$ is the phase-to-ground voltage magnitude square at node $n$. $p_{o_{l}}, q_{o_{l}}$ are the active and reactive power losses of branch $l$. Constraints \eqref{eq:p_balance_chap1}-\eqref{eq:vsvrsin} make the o-ACOPF model valid for both radial and meshed power networks. Equations (\ref{eq:VsVr_chap2})-(\ref{eq:vsvrsin}) are derived by taking the magnitude and phase angle of {the} voltage drop phasor along branch $l$ respectively. $v_{s_{l}}, v_{r_{l}}$ are the sending- and receiving- end phase-to-ground voltages of branch $l$. $V_{s(r)_{l}}=v^{2}_{s(r)_{l}}$ are voltage magnitude squares. $\theta_{l}=\theta_{s_{l}}-\theta_{r_{l}}$ is the phase angle difference between the sending- and receiving- end voltages of branch $l$. $\theta_{s_{l}}, \theta_{r_{l}}$ are the phase angles of sending- and receiving- end phase-to-ground voltages of branch $l$. To guarantee this derivation is valid, we assume $(\theta^{min}_l,\theta^{max}_l) \subseteq (-\frac{\pi}{2}, \frac{\pi}{2})$. $i_{s_{l}},i_{r_{l}}$ are the non-measurable sending- and receiving- end currents of branch $l$. $\widetilde{i}_{s_{l}},\widetilde{i}_{r_{l}}$ are the measurable sending- and receiving- end currents of branch $l$. $\widetilde{K}_l$ is the actual ampacity of branch $l$ which is provided {by} the branch manufacturer. Equations (\ref{eq:ploss})-(\ref{eq:qloss}) represent active power and reactive power losses. Constraints (\ref{eq:vbound})-(\ref{eq:qbound}) are bounds for voltage magnitude, voltage phase angle difference, nodal active power generation and nodal reactive power generation. $p^{min}_{n}, p^{max}_{n}$ are the lower and upper bounds of $p_{n}$. $q^{min}_{n}, q^{max}_{n}$ are the lower and upper bounds of $q_{n}$. $\theta^{min}_{l}, \theta^{max}_{l}$ are the lower and upper bounds of $\theta_{l}$. This model is nonconvex because of the nonconvex constraints \eqref{eq:vsvrsin}, \eqref{eq:veq}, \eqref{eq:ploss} and \eqref{eq:qloss}. It is hard for current available nonlinear programming solvers to efficiently find the global optimal solution of this nonconvex optimization model. 

It is worth to mention that including $\theta_n$ as one of the model variables and using \eqref{eq:def_thetal} to define $\theta_l$ make our model implicitly satisfy the cyclic constraint (Kirchoff's voltage law for AC circuit) of meshed power networks:
\begin{align}
\sum_{l\in \mathcal{C}} \theta_l=\sum_{l\in \mathcal{C}} (\theta_{s_l}-\theta_{r_l})=0\; mod\; 2\pi \label{eq:cons_cyclic}
\end{align}
\end{subequations}
\begin{figure}[!htbp]
\vspace{-0.8cm}
\centering
\includegraphics[width=8.5cm]{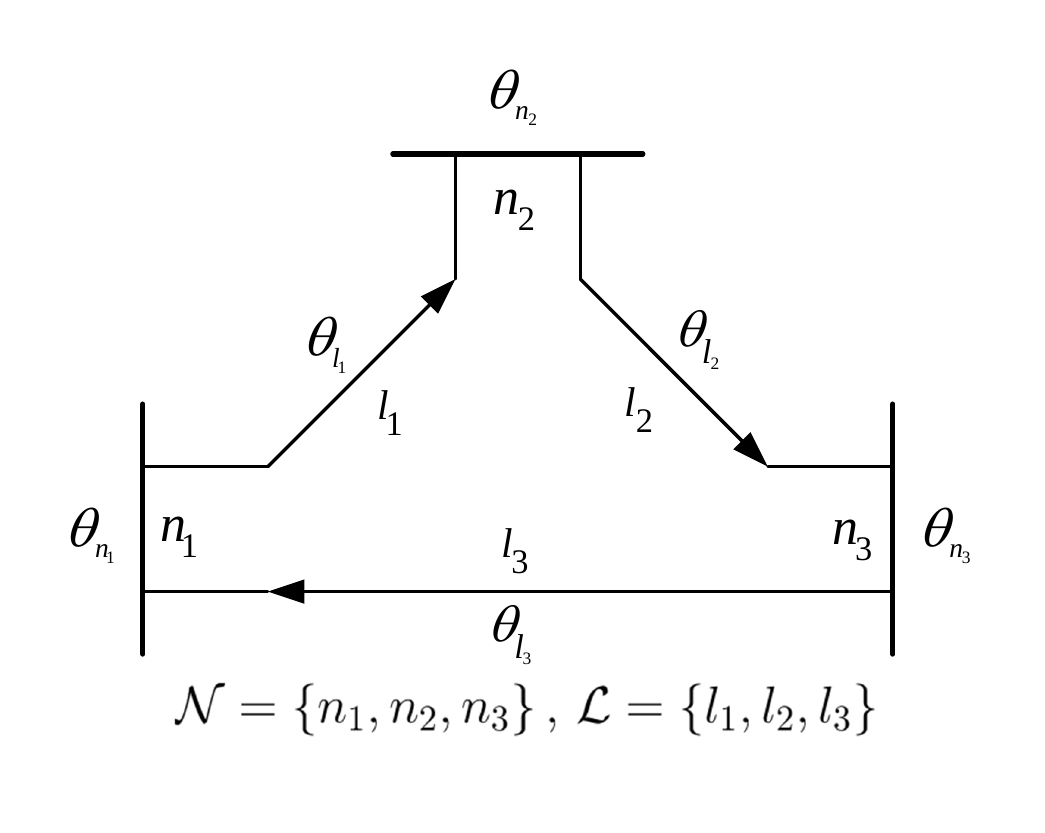}
\vspace{-0.9cm}
\caption{The implicit expression of cyclic constraint}\label{fig:loop_zero}
\end{figure}

Where $\mathcal{C}$ is the set of all cycles or closed loops in the meshed power networks. This means we do not need to include the cyclic constraint \eqref{eq:cons_cyclic} explicitly in our model. This approach is used and validated by our work in \cite{DLMP_me, SOCACOPF_relax_fea_me}. An illustrative example is provided in Fig. \ref{fig:loop_zero}. In this example, since $\theta_{l_1}=\theta_{n_1}-\theta_{n_2}$, $\theta_{l_2}=\theta_{n_2}-\theta_{n_3}$ and $\theta_{l_3}=\theta_{n_3}-\theta_{n_1}$, the cyclic constraint $\theta_{l_1}+\theta_{l_2}+\theta_{l_3}=0\;mod\;2\pi$ is implicitly satisfied. This reasoning holds for cycles or closed loops with any number of nodes and branches.

\subsection{Second-Order Cone Relaxation}
The SOC-ACOPF model is derived using branch sending-end power flows and voltage phase angle variables.  Note that in the derived model, voltage magnitude square variables are {used} (voltage magnitude can be recovered from the model by taking the square root of the solutions of the voltage magnitude square variables). {Our} SOC-ACOPF is formulated in optimization model-(\ref{SOPF-P}). The numerical performance of this model compared with other SOCP based ACOPF models can be found at \cite{SOCACOPF_relax_fea_me}. 
\begin{subequations}\label{SOPF-P}
\begin{align}
    &\underset{\Omega_{SOC-ACOPF}}{\text{Minimize}} {\hspace{10pt}}  f(\Omega_{SOC-ACOPF})\label{eq:objective}\\
    &\text{subject to} \;\;\; \eqref{eq:p_balance_chap1}-\eqref{eq:VsVr_chap2}, \eqref{eq:def_thetal},\eqref{eq:vbound}-\eqref{eq:qbound}\nonumber\\
    &K_{o_{l}}\geq{q}_{o_{l}} \geq \frac{ p^{2}_{s(r)_{l}}+q^{2}_{s(r)_{l}}}{V_{s(r)_{l}}} X_{l},\; \forall l\in \mathcal{L}\label{eq:plosscone}\\
    &p_{o_{l}}X_{l}=q_{o_{l}}R_{l},\; \forall l\in \mathcal{L}\label{eq:pqloss}\\
    &\theta_{l}=X_{l}p_{s_{l}}-R_{l}q_{s_{l}},\; \forall l\in \mathcal{L}\label{eq:sinxprq}
\end{align}
\end{subequations}
Where $\Omega_{SOC-ACOPF}=\{p_{n},q_{n},p_{s_{l}},q_{s_{l}},p_{o_{l}},q_{o_{l}},V_{n},$ $\theta_{l},\theta_{n}\}$ $\in \Re^{m-\left | \mathcal{N} \right |}$ is the set of model variables. Since we only use voltage magnitude square variable $V_n$ in this SOC-ACOPF model, the voltage magnitude constraint \eqref{eq:vbound} is modified as $V_{n}\in (V^{min}_{n},\, V^{max}_{n}) = ((v^{min}_{n})^2,\, (v^{max}_{n})^2)$ in the SOC-ACOPF model-\eqref{SOPF-P}. Constraints \eqref{eq:plosscone}-\eqref{eq:pqloss} represent active power and reactive power losses. $K_{o_{_l}}$ are the equivalent ampacity constraint for power losses of branch $l$. These constraints are relaxed from constraints \eqref{eq:ploss}-\eqref{eq:qloss} in the o-ACOPF model-\eqref{NOPF}. The left side of constraint \eqref{eq:plosscone} bounds ${q}_{o_{l}}$, which equivalently bounds the ampacity of branch $l$ as explained in the next section. Note even though the right side of constraint \eqref{eq:plosscone} is not in the standard form of SOCP, it can be equivalently transformed to a rotated SOCP constraint as ${q}_{o_{l}}{V_{s(r)_{l}}} \geq ({ p^{2}_{s(r)_{l}}+q^{2}_{s(r)_{l}}}) X_{l}$. We keep the original format of the right side of constraint \eqref{eq:plosscone} in the SOC-ACOPF model-\eqref{SOPF-P} because the physical meaning is more clear.
Equation (\ref{eq:sinxprq}) is derived from the nonconvex constraint \eqref{eq:vsvrsin}. This derivation is based on the assumption $v_{s_{l}}v_{r_{l}} sin \theta_{l} \approx \theta_{l}$.
The left side of equation \eqref{eq:sinxprq} can also be derived by the first-order Taylor series expansion of $v_{s_{l}}v_{r_{l}}\sin\theta_{l}$ at $v_{s_{l}}=v_{r_{l}}=1$ and $\theta_{l}=0$. We will show later that using \eqref{eq:sinxprq} actually relaxes the ACOPF model when $(v^{min}_n, v^{max}_n) = (0.9, 1.1)$ and $(\theta^{min}_l, \theta^{max}_l)  = (-\frac{\pi}{2}, \frac{\pi}{2})$. It is worth to mention that even though constraint \eqref{eq:sinxprq} is the only constraint associated with the voltage phase angle variable, it should be solved jointly with other constraints in order to guarantee the feasibility of $\theta_l, p_{ol}, q_{o,l}$.
\section{Deriving the {Branch} Ampacity constraint for the Power Losses}
Since the actual measurable power flows $(\widetilde{p}_{s_{l}},\widetilde{q}_{s_{l}})$ and current $\widetilde{i}_{s(r)_{l}}$ are different from the power flow variables $(p_{s_{l}},q_{s_{l}})$ and $i_{s(r)_{l}}$ that we are using in the SOC-ACOPF model-\eqref{SOPF-P}, it is necessary to derive the gap between $\widetilde{i}_{s(r)_{l}}$ and $i_{s(r)_{l}}$ to derive $K_{o_{l}}$ according to the known parameter $\widetilde{K}_l$.
From Fig. \ref{fig:trans_line_pi}, we have:
\begin{subequations}\label{appgap}
\begin{align}
&\widetilde{p}_{s(r)_{l}}=p_{s(r)_{l}} \label{eq:psps}\\
&\widetilde{q}_{s(r)_{l}}=q_{s(r)_{l}}-q_{cs(r)_{l}}\label{eq:qsqs}\\
&q_{cs(r)_{l}}=V_{s(r)_{l}}B_{s(r)_{l}}\label{eq:qcsvb}
\end{align}
Where $q_{cs(r)_{l}}$ is the reactive power injection from the sending- (receiving-) end shunt capacitance of the branch $l$, $B_{s(r)_{l}}$ is the shunt susceptance. The {branch} ampacity constraint is:
\begin{align}
&\left \| \widetilde{i}_{s(r)_{l}} \right \|^{2}=\frac{\widetilde{p}^{2}_{s(r)_{l}}+\widetilde{q}^{2}_{s(r)_{l}}}{V_{s(r)_{l}}}\leq \widetilde{K}_l\label{eq:psqsk}
\end{align}
From \eqref{eq:psps}-\eqref{eq:psqsk}, we can derive the gap $\Delta^2 I$ between $\left \| i_{s(r)_{l}} \right \|^{2}$ and $\left \| \widetilde{i}_{s(r)_{l}} \right \|^{2}$ {as}:
\vspace{-0.5cm}
\begin{align}
\Delta^2 I&=\left \| i_{s(r)_{l}} \right \|^{2}-\left \| \widetilde{i}_{s(r)_{l}} \right \|^{2}\nonumber\\
&=\frac{-q^{2}_{cs(r)_{l}}+2q_{s(r)_{l}}q_{cs(r)_{l}}}{V_{s(r)_{l}}} \nonumber\\
&=\frac{-V^{2}_{s(r)_{l}}B^2_{s(r)_{l}}+2q_{s(r)_{l}}V_{s(r)_{l}}B_{s(r)_{l}}}{V_{s(r)_{l}}} \nonumber\\
&=-V_{s(r)_{l}}B^2_{s(r)_{l}}+2q_{s(r)_{l}}B_{s(r)_{l}}
\end{align}
The branch ampacity constraint \eqref{eq:psqsk} is equivalent to:
\begin{align}
&i^2_{s(r)_{l}}=\frac{p^2_{s(r)_{l}}+q^2_{s(r)_{l}}}{V_{s(r)_{l}}}\leq K_l= \widetilde{K}_l+\Delta^2 I\label{eq:psqsk2}
\end{align}
The reactive power losses upper bounds $K_{ol}$ can be quantified as:
\begin{align}
K_{o_{l}}&= K_{l}X_{l}\nonumber\\
&=(\widetilde{K}_l+\Delta^2 I)X_{l}\nonumber\\
&=(\widetilde{K}_l-V_{s(r)_{l}}B^2_{s(r)_{l}}+2q_{s(r)_{l}}B_{s(r)_{l}})X_{l}\label{eq:kol}
\end{align}
\end{subequations}
Note equation \eqref{eq:kol} is linear. So, if we use the expression of $K_{ol}$ from \eqref{eq:kol} in the constraint \eqref{eq:plosscone}, the SOC-ACOPF model-\eqref{SOPF-P} is still convex and we avoid any approximation on the branch ampacity constraint.

It is worth to mention that using $K_{o_{l}}$ to {limit} the upper bound of power losses $p_{o_l}$ in \eqref{eq:plosscone} is more realistic than constraining the power flows $(\widetilde{p}_{s_{l}},\widetilde{q}_{s_{l}})$ or $(p_{s_{l}},q_{s_{l}})$ in the way of $\widetilde{p}^{2}_{s_{l}}+\widetilde{q}^{2}_{s_{l}}\leq S_l$ or ${p}^{2}_{s_{l}}+{q}^{2}_{s_{l}}\leq S_l$ (where $S_l$ is the maximum branch power flow). This is because:
\begin{itemize}
  \item (i) the branch ampacity is given by the manufacturers in the form of maximum current $\widetilde{i}^{max}_{s(r)_{l}}=\sqrt{\widetilde{K}_{l}}$.
  \item (ii) the temperature increase of the branch (which lead to insulation degrading) is actually caused by the power losses due to the current which is the typical variable constrained by the branch manufacturer.
  \item (iii) the voltage amplitudes $v_{s(r)_{l}}$ of both ends of the branch are varying during the power network operations. The maximum allowed power capacity $S_l=(v_{s(r)_{l}}\widetilde{i}^{max}_{s(r)_{l}})^2$ of the branch would then depend on the nodal voltage amplitudes at the branch ends.
\end{itemize}

\section{Properties of the SOC-ACOPF Model}
In order to rigorously prove the important properties of the SOC-ACOPF mdoel-\eqref{SOPF-P} and the o-ACOPF model-\eqref{NOPF}, we firstly give the definitions of some key words. Some well-defined terms in the control and optimization community are also re-stated here for sake of readability. For ease of illustration, all the mentioned optimization models in this paper are assumed as minimization problems i.e. we consider minimizing the objective functions. Maximization problems can be equivalently transformed to minimization problems by minimizing the negative of the objective functions (of the maximization problems). If multiple optimization models are mentioned in one definition or theorem here, unless otherwise specified, we always mean that they have the same objective function. {Note here we use the symbols $\bold{M}$ and $\bold{N}$ to denote two general optimization models. Our ACOPF model}-{\eqref{NOPF} and SOC-ACOPF model}-{\eqref{SOPF-P} can be regarded as examples of the optimization models $\bold{M}$ and $\bold{N}$.}

\begin{table*}[!htbp]
  \centering
  \caption{Overview of the Proposed Theorems}
    \resizebox{\textwidth}{!}{%
    \begin{tabular}{|c|l|}\hline
    Theorem & \hspace{10cm}Property \\\hline
    1     & The SOC-ACOPF model-(2) is a relaxation of the o-ACOPF model-(1). \\ \hline
    2     & Necessary condition to recover feasible solution of the o-ACOPF model-(1) from the solution of the SOC-ACOPF model-(1). \\ \hline
    3     & The SOC-ACOPF model-(2) with the additional constraint \eqref{eq:necerecover} is still a relaxation of the o-ACOPF model-(1). \\ \hline
    4     & Recover the global optimal solution of the o-ACOPF model-(1) from the solution of the SOC-ACOPF model-(2). \\ \hline
    5     & The parametric optimal value functions of the o-ACOPF model-(1) and the SOC-ACOPF model-(2) are monotonic. \\ \hline
    6     & Tight solution of the SOC-ACOPF model-(2) always exist. \\ \hline
    \end{tabular}}%
  \label{tab:theorem_oview}%
\end{table*}%
\begin{definition}
Feasible solution. A feasible solution $\Omega_{0,\bold{M}}$ of an optimization model {$\bold{M}$} is the solution satisfying jointly all the constraints of {$\bold{M}$}.
\end{definition}
\begin{definition} 
Feasible region. The feasible region $\mathcal{R}_{\bold{M}}$ of an optimization model {$\bold{M}$} is the set of all feasible solutions of {$\bold{M}$}.
\end{definition}
\begin{definition}
Global optimal solution. A global optimal solution $\Omega^*_M \in \mathcal{R}_{\bold{M}}$ of an optimization model {$\bold{M}$} is the feasible solution of {$\bold{M}$} which gives the minimal (for minimization problem) objective function value i.e. $\{f(\Omega^*_M)\leq f(\Omega_{0,\bold{M}}),\,\forall\, \Omega_{0,\bold{M}} \in \mathcal{R}_{\bold{M}}\}$.
\end{definition}
\begin{definition}
Tight solution. A tight solution of an optimization model {$\bold{M}$} is the feasible solution making some inequality constraints of {$\bold{M}$} active (binding). In this paper, we are considering the tight solution for the right side of the SOCP constraint \eqref{eq:plosscone} in the SOC-ACOPF model-\eqref{SOPF-P}.
\end{definition}
\begin{definition}
Relaxation. An optimization model {$\bold{N}$} is said to be a relaxation or a relaxed model of optimization model $\bold{M}$ if and only if the feasible region of {$\bold{M}$} is a subset of the feasible region of $\bold{N}$ i.e. $\mathcal{R}_{\bold{M}} \subseteq \mathcal{R}_{\bold{N}}$.
\end{definition}
\begin{definition}
Approximation. An optimization model {$\bold{N}$} is said to be an (strict) approximation or an approximated model of optimization model {$\bold{M}$} if and only if the feasible regions of {$\bold{M}$} and {$\bold{N}$} are approximately equal but neither one is a subset of another i.e. $\{\mathcal{R}_{\bold{N}}\approx \mathcal{R}_{\bold{M}}, \mathcal{R}_{\bold{N}} \nsubseteq \mathcal{R}_{\bold{M}}, \mathcal{R}_{\bold{M}} \nsubseteq \mathcal{R}_{\bold{N}}\}$.
\end{definition}
\begin{subequations}
\begin{definition}
Relaxation gap. For the SOC-ACOPF model-\eqref{SOPF-P}, the relaxation gap of active power loss $Gap^{po}_{l}$ is defined as:
\begin{align}
Gap^{po}_{l}:=p_{o_{l}}-\frac{p^2_{s_{l}}+q^{2}_{s_{l}}}{V_{s_{l}}}R_{l},\; \forall l\in \mathcal{L}\\
\end{align}

The relaxation gap of reactive power loss $Gap^{qo}_{l}$ is defined as:
\begin{align}
Gap^{qo}_{l}:=q_{o_{l}}-\frac{p^2_{s_{l}}+q^{2}_{s_{l}}}{V_{s_{l}}}X_{l},\; \forall l\in \mathcal{L}
\end{align}
\end{definition}
To give an overview of the proposed theorems, we summarize {the} theorems in Table \ref{tab:theorem_oview}. Theorem 4 is a based on theorem 1, theorem 2 and thorem 3. Theorem 6 is based on theorem 5.

\begin{theorem}\label{theom:relax}
Assume $(v^{min}_n, v^{max}_n) = (0.9, 1.1)$ and $(\theta^{min}_l, \theta^{max}_l)$ $=(-\frac{\pi}{2}, \frac{\pi}{2})$, replacing \eqref{eq:vsvrsin} by \eqref{eq:sinxprq} relaxes the o-ACOPF model-\eqref{NOPF} i.e. the SOC-ACOPF model-\eqref{SOPF-P} is a relaxation of the o-ACOPF model-\eqref{NOPF} (rather than an approximation of the o-ACOPF model-\eqref{NOPF}).
\end{theorem}
\begin{proof}
We prove this theorem by showing that any point in the feasible region of o-ACOPF model-\eqref{NOPF} can always be mapped to a point located in the feasible region of the SOC-ACOPF model-\eqref{SOPF-P}. However, the reverse statement does not hold. In other words, some feasible solutions of the SOC-ACOPF model-\eqref{SOPF-P} are not feasible for the o-ACOPF model-\eqref{NOPF}. This means the feasible region of the SOC-ACOPF model-\eqref{SOPF-P} covers the feasible region of the o-ACOPF mode-\eqref{NOPF}.

Suppose $\Omega_{0,ACOPF}=\{p_{0,n},q_{0,n},p_{0,s_{l}},q_{0,s_{l}},p_{0,o_{l}},$ $q_{0,o_{l}},V_{0,n},v_{0,n},\theta_{0,l},\theta_{0,n}\} \in \Re^{m}$ is one feasible solution of the o-ACOPF model-\eqref{NOPF}. From \eqref{eq:vsvrsin} we have:
\begin{align}
v_{0,s_{l}}v_{0,r_{l}}\sin\theta_{0,l}=X_{l}p_{0,s_{l}}-R_{l}q_{0,s_{l}},\; \forall l\in \mathcal{L}\label{eq:vsvrsin0}
\end{align}
We can map $\Omega_{0,ACOPF}$ to a feasible solution $\Omega_{1,SOC_ACOPF}=\{p_{1,n},q_{1,n},p_{1,s_{l}},q_{1,s_{l}},p_{1,o_{l}},q_{1,o_{l}},$ $V_{1,n},\theta_{1,l},\theta_{1,n}\} \in \Re^{m-\left | \mathcal{N} \right |}$ of the SOC-ACOPF model-\eqref{SOPF-P} as:
\begin{align}
&\left \{ \Omega_{1,SOC-ACOPF}\setminus \theta_{1,l} \right \} :=\left \{ \Omega_{0,ACOPF}\setminus (\theta_{0,l}, v_{0,n})\right \}\label{eq:omega1}\\
&\theta_{1,l}:=v_{0,s_{l}}v_{0,r_{l}}\sin\theta_{0,l},\; \forall l\in \mathcal{L}\label{eq:vsvrsin1}
\end{align}
Since $v_{0,s_{l}}v_{0,r_{l}}\sin\theta_{0,l}\in (-1.21, 1.21)\subset (-\frac{\pi}{2}, \frac{\pi}{2})$, equation \eqref{eq:vsvrsin1} is always mappable. Note $\theta_{1,l}$ is not necessarily equal to $\theta_{0,l}$.
On the other hand, mapping $\Omega_{1}$ to $\Omega_{0}$ ($v_{0,s_{l}}v_{0,r_{l}}\sin\theta_{0,l}=\theta_{1,l}$) is not feasible when $\theta_{1,l}>1.21$ or $\theta_{1,l}<-1.21$.
\end{proof}
Theorem \ref{theom:relax} shows that the feasible region of SOC-ACOPF model-\eqref{SOPF-P} covers all the feasible region of o-ACOPF model-\eqref{NOPF}. It is worth to mention that the assumptions of $(v^{min}_n, v^{max}_n) = (0.9, 1.1)$ and $(\theta^{min}_l, \theta^{max}_l)  = (-\frac{\pi}{2}, \frac{\pi}{2})$ are the practical operational requirements for power system. These assumptions can be relaxed to some {degree} while Theorem \ref{theom:relax} is still valid.

\begin{theorem}\label{theom:recover}
If $(\theta^{min}_l,\theta^{max}_l) \subseteq (-\frac{\pi}{2}, \frac{\pi}{2})$ and $\theta^{min}_l=-\theta^{max}_l$, the necessary condition to recover (map) a feasible solution of the o-ACOPF model-\eqref{NOPF} from the (optimal) solution $\Omega_{1,SOC-ACOPF}$ of SOC-ACOPF model-\eqref{SOPF-P} is:
\begin{align}
V_{s_{l}}V_{r_{l}}sin^2(\theta^{max}_l) \geq \theta^2_l ,\; \forall l\in \mathcal{L} \label{eq:necerecover}
\end{align}
Note constraint \eqref{eq:necerecover} is conic and thus convex.
\end{theorem}
\begin{proof}
If a feasible solution $\Omega_0$ of the o-ACOPF model-\eqref{NOPF} is recovered (mapped) from the (optimal) solution $\Omega_{1}$ of the SOC-ACOPF model-\eqref{SOPF-P}:
\begin{align}
&\left \{ \Omega_{0}\setminus (\theta_{0,l}, v_{1,n}) \right \} :=\left \{ \Omega_{1}\setminus (\theta_{1,l})\right \} \label{eq:recover1}\\
&v_{0,n}:=\sqrt{V_{1,n}},\; \forall n\in N \label{eq:recover2}\\
&\theta_{0,l}:=arcsin(\frac{X_{l}p_{s_{1,l}}-R_{l}q_{s_{1,l}}}{v_{1,s_{l}}v_{1,r_{l}}})\nonumber\\
&\hspace{18pt}=arcsin(\frac{\theta_{1,l}}{v_{1,s_{l}}v_{1,r_{l}}}),\; \forall l\in \mathcal{L}\label{eq:thetarecover}
\end{align}
Since $sin(\theta_{1,l})$ is monotonic in $(\theta^{min}_l,\theta^{max}_l) \subseteq (-\frac{\pi}{2}, \frac{\pi}{2})$, equation \eqref{eq:thetarecover} implies:
\begin{align}
 sin(\theta^{min}_l) \leq \frac{\theta_{1,l}}{v_{s_{1,l}}v_{r_{1,l}}} \leq sin(\theta^{max}_l) ,\; \forall l\in \mathcal{L}\label{eq:thetarecover2}
\end{align}
Considering $sin^2(\theta^{min}_l)=sin^2(\theta^{max}_l)$ when $\theta^{min}_l=-\theta^{max}_l$, constraint \eqref{eq:thetarecover2} is equivalent to:
\begin{align}
\frac{\theta^2_{1,l}}{V_{s_{l}}V_{r_{l}}sin^2(\theta^{max}_l)} \leq 1  ,\; \forall l\in \mathcal{L}
\end{align}
Or equivalently, $V_{s_{1,l}}V_{r_{1,l}}sin^2(\theta^{max}_l) \geq \theta^2_{1,l} $.
\end{proof}
Theorem \ref{theom:recover} shows that if we add the convex constraint \eqref{eq:necerecover} to the SOC-ACOPF model-\eqref{SOPF-P}, we can improve the solution feasibility towards the o-ACOPF model-\eqref{NOPF}. Please note that equations \eqref{eq:recover1}-\eqref{eq:thetarecover} is only one way to recover the feasible solution. Condition \eqref{eq:necerecover} is necessary to recover the feasible solution by using this approach.  There can be other feasible solution recovery approaches which do not require condition \eqref{eq:necerecover}.

\begin{theorem}\label{theom:relax2}
The SOC-ACOPF model-\eqref{SOPF-P} with the additional constraint \eqref{eq:necerecover} is a relaxation of the o-ACOPF model-\eqref{NOPF}.
\end{theorem}
\begin{proof}
We can use the same procedure in proving theorem \ref{theom:relax} to prove theorem \ref{theom:relax2} i.e. any feasible solution of the o-ACOPF model-\eqref{NOPF} can be mapped to a feasible solution of the SOC-ACOPF model-\eqref{SOPF-P} with {the} additional constraint \eqref{eq:necerecover}. However, the reverse statement is not true when there is at least one $\hat{l} \in \mathcal{L}$ (in the feasible solution of the SOC-ACOPF model-\eqref{SOPF-P} with the additional constraint \eqref{eq:necerecover}) such that ${q}_{o_{\hat{l}}} > \frac{ p^{2}_{s(r)_{\hat{l}}}+q^{2}_{s(r)_{\hat{l}}}}{V_{s(r)_{\hat{l}}}} X_{\hat{l}}$ (which fails to satisfy constraints \eqref{eq:qloss} in the o-ACOPF model-\eqref{NOPF} obviously).
\end{proof}
Theorem \ref{theom:relax2} shows that the SOC-ACOPF model-\eqref{SOPF-P} with the additional constraint \eqref{eq:necerecover} still covers the feasible region of the o-ACOPF model-\eqref{NOPF}.

\begin{theorem}\label{theom:global}
If the optimal solution $\Omega^{*}$ of the SOC-ACOPF model-\eqref{SOPF-P} with {the} additional constraint \eqref{eq:necerecover} gives tight solution such that ${q}^{*}_{o_{l}} = \frac{ p^{*2}_{s(r)_{l}}+q^{*2}_{s(r)_{l}}}{V^{*}_{s(r)_{l}}} X_{l},\; \forall l \in \mathcal{L}$, the global optimal solution $\Omega^{*}_{0}$ of the o-ACOPF model-\eqref{NOPF} is $\Omega^{*}_{0}:=\left \{ \Omega^{*}\setminus (\theta^{*}_{l},V^{*}_{n}) \right \} \cup \left \{ v^{*}_{0,n}:=\sqrt{V^{*}_{n}},\, \theta^{*}_{0,l}:=arcsin(\frac{\theta^{*}_{l}}{v^{*}_{s_{l}}v^{*}_{r_{l}}}) \right \}$.
\end{theorem}
\begin{proof}
Theorem \ref{theom:global} is a direct result of theorem \ref{theom:relax}, theorem \ref{theom:recover} and theorem \ref{theom:relax2} considering the optimal objective solution $f^{*}$ of the SOC-ACOPF model-\eqref{SOPF-P} is always a lower bound for the optimal objective solution $f^{*}_{0}$ of the o-ACOPF model-\eqref{NOPF}.
\end{proof}

\begin{theorem}\label{theom:mono}
Define the (parametric) optimal value functions:
\begin{align}
  &f^{*}(p_{d_{n}},q_{d_{n}}):={min}\; f(\Omega, p_{d_{n}},q_{d_{n}})\nonumber\\
  &\hspace{78pt} s.t.\; \left \{ \Omega \in \mathcal{R}_{ACOPF} \right \} \\
  &f^{*}(p_{d_{n}},q_{d_{n}}):={min}\; f(\Omega, p_{d_{n}},q_{d_{n}})\nonumber\\
  &\hspace{78pt} s.t.\; \left \{ \Omega \in \mathcal{R}_{SOC-ACOPF} \right \}
\end{align}
Where $\mathcal{R}_{ACOPF}$ and $\mathcal{R}_{SOC-ACOPF}$ are the feasible regions of the o-ACOPF model-\eqref{NOPF} and {the} SOC-ACOPF model-\eqref{SOPF-P} correspondingly.
If the objective function $f$ is a monotonically increasing function of $(p_n ,q_n)$  i.e. $f(p_{1,n}, q_{1,n}) \leq f(p_{2,n}, q_{2,n})$ for $(p_{1,n}, q_{1,n}) \leq
(p_{2,n} ,q_{2,n})$, the optimal value functions $f^{*}$ are monotonic i.e. $f^{*}(p_{d_{1,n}},q_{d_{1,n}})\leq f^{*}(p_{d_{2,n}},q_{d_{2,n}})$ for $(p_{d_{1,n}},q_{d_{1,n}})\leq (p_{d_{2,n}},q_{d_{2,n}})$ (assuming the o-ACOPF {model-}\eqref{NOPF} and {the} SOC-ACOPF model-\eqref{SOPF-P} are feasible for $(p_{d_{1,n}},q_{d_{1,n}})$ and $(p_{d_{2,n}},q_{d_{2,n}})$).
\end{theorem}
\begin{proof}
We prove theorem \ref{theom:mono} by contradiction. Suppose the optimal solutions and objectives of the SOC-ACOPF {model-}\eqref{SOPF-P} (or {the} o-ACOPF {model-}\eqref{NOPF}) at $(p_{d_{1,n}},q_{d_{1,n}})$ and $(p_{d_{2,n}},q_{d_{2,n}})$ are $\Omega^{*}_{1}=\{p^{*}_{1,n},q^{*}_{1,n},p^{*}_{1,s_{l}},q^{*}_{1,s_{l}},p^{*}_{1,o_{l}},q^{*}_{1,o_{l}},V^{*}_{1,n},\theta^{*}_{1,l}\} \in \Re^{m-\left | \mathcal{N} \right |}, f^{*}_{1}$ and $\Omega^{*}_{2}=\{p^{*}_{2,n},q^{*}_{2,n},p^{*}_{2,s_{l}},q^{*}_{2,s_{l}},p^{*}_{2,o_{l}},$ $q^{*}_{2,o_{l}},V^{*}_{2,n},\theta^{*}_{2,l}\} \in \Re^{m-\left | \mathcal{N} \right |}, f^{*}_{2}$. If $f^{*}_{1}>f^{*}_{2}$ for $(p_{d_{1,n}},q_{d_{1,n}})\leq (p_{d_{2,n}},q_{d_{2,n}})$, since $f$ is monotonic, there must be at least one $\hat{n} \in N$ such that $(p^{*}_{1,\hat{n}},q^{*}_{1,\hat{n}})>(p^{*}_{2,\hat{n}},q^{*}_{2,\hat{n}})$. We can construct a feasible solution $\Omega^{'}_{1}$ of the SOC-ACOPF {model-}\eqref{SOPF-P} (or {the} o-ACOPF {model-}\eqref{NOPF}) at $(p_{d_{1,n}},q_{d_{1,n}})$ (by using $\Omega^{*}_{2}$) as:
\begin{align}
  \left \{ \Omega^{'}_{1}\setminus (p^{'}_{1,\hat{n}},q^{'}_{1,\hat{n}}) \right \} :=\left \{ \Omega^{*}_{2}\setminus (p^{*}_{2,\hat{n}},q^{*}_{2,\hat{n}})\right \}
\end{align}
From \eqref{eq:p_balance_chap1}-\eqref{eq:q_balance_chap1}, we have:
\begin{align}
  &p^{*}_{2,\hat{n}}-p_{d_{2,\hat{n}}}=\sum_{l}({A^{+}_{\hat{n}l}}p^{*}_{s_{2,l}}-{A^{-}_{\hat{n}l}}p^{*}_{o_{2,l}})+G_{\hat{n}}V^{*}_{2,\hat{n}} \label{eq:p_balance3}\\
  &q^{*}_{2,\hat{n}}-q_{d_{2,\hat{n}}}=\sum_{l}({A^{+}_{\hat{\hat{n}}l}}q^{*}_{s_{2,l}}-{A^{-}_{\hat{n}l}}q^{*}_{o_{2,l}})-B_{\hat{n}}V^{*}_{2,\hat{n}}\label{eq:q_balance3}
\end{align}
We know the feasible solution of $(p^{'}_{1,\hat{n}},q^{'}_{1,\hat{n}})$ must satisfy: 
\begin{align}
  &p^{'}_{1,\hat{n}}-p_{d_{1,\hat{n}}}=\sum_{l}({A^{+}_{\hat{n}l}}p^{*}_{s_{2,l}}-{A^{-}_{\hat{n}l}}p^{*}_{o_{2,l}})+G_{\hat{n}}V^{*}_{2,\hat{n}}\label{eq:p_balance4}\\
  &q^{'}_{1,\hat{n}}-q_{d_{1,\hat{n}}}=\sum_{l}({A^{+}_{\hat{n}l}}q^{*}_{s_{2,l}}-{A^{-}_{\hat{n}l}}q^{*}_{o_{2,l}})-B_{\hat{n}}V^{*}_{2,\hat{n}}\label{eq:q_balance4}
\end{align}
Substituting \eqref{eq:p_balance4}-\eqref{eq:q_balance4} to \eqref{eq:p_balance3}-\eqref{eq:q_balance3}, we have:
\begin{align}
  &p^{'}_{1,\hat{n}}-p_{d_{1,\hat{n}}}=p^{*}_{2,\hat{n}}-p_{d_{2,\hat{n}}}\\
  &q^{'}_{1,\hat{n}}-q_{d_{1,\hat{n}}}=q^{*}_{2,\hat{n}}-q_{d_{2,\hat{n}}}
\end{align}
Which yield:
\begin{align}
  &p^{'}_{1,\hat{n}}=p^{*}_{2,\hat{n}}+p_{d_{1,\hat{n}}}-p_{d_{2,\hat{n}}}\label{eq:p_balance5}\\
  &q^{'}_{1,\hat{n}}=q^{*}_{2,\hat{n}}+q_{d_{1,\hat{n}}}-q_{d_{2,\hat{n}}}\label{eq:q_balance5}
\end{align}
So $\Omega^{'}_{1}=\left \{ \Omega^{*}_{2}\setminus (p^{*}_{2,\hat{n}},q^{*}_{2,\hat{n}})\right \}\cup (p^{'}_{1,\hat{n}},q^{'}_{1,\hat{n}})$ is feasible for the SOC-ACOPF {model-}\eqref{SOPF-P} (or {the} o-ACOPF {model-}\eqref{NOPF}). Because $(p_{d_{1,n}},q_{d_{1,n}})\leq (p_{d_{2,n}},q_{d_{2,n}})$, equations \eqref{eq:p_balance5}-\eqref{eq:q_balance5} imply $(p^{'}_{1,\hat{n}},q^{'}_{1,\hat{n}})\leq(p^{*}_{2,\hat{n}},q^{*}_{2,\hat{n}})$. This means the corresponding objective function values satisfy $f^{'}_{1}\leq f^{*}_{2}<f^{*}_{1}$ (note $(p^{'}_{1,n},q^{'}_{1,n})=(p^{*}_{2,n},q^{*}_{2,n}),\; \forall n\neq \hat{n}$) which contradicts the assumption that $f^{*}_{1}$ is the optimal objective solution.  So $f^{*}_{1}>f^{*}_{2}$ cannot hold for $(p_{d_{1,n}},q_{d_{1,n}})\leq (p_{d_{2,n}},q_{d_{2,n}})$. This completes the proof.
\end{proof}

\begin{theorem}\label{theom:pd}
Assuming $(\theta^{min}_l,\theta^{max}_l) \subseteq (-\frac{\pi}{2}, \frac{\pi}{2})$ and the objective function $f$ is monotonically increasing for $(p_{n},q_{n})$, if the optimal solution $\Omega^{*}$ of the SOC-ACOPF model-\eqref{SOPF-P} (with the additional constraint \eqref{eq:necerecover}) at $(p_{d_{n}},q_{d_{n}})$ has positive relaxation gap i.e. ${q}^{*}_{o_{\hat{l}}} > \frac{ p^{*2}_{s(r)_{\hat{l}}}+q^{*2}_{s(r)_{\hat{l}}}}{V^{*}_{s(r)_{\hat{l}}}} X_{\hat{l}}$ or ${p}^{*}_{o_{\hat{l}}} > \frac{ p^{*2}_{s(r)_{\hat{l}}}+q^{*2}_{s(r)_{\hat{l}}}}{V^{*}_{s(r)_{\hat{l}}}} R_{\hat{l}}$ for some $\hat{l} \in \mathcal{L}$, there exists $(p^{'}_{d_{n}},q^{'}_{d_{n}})>(p_{d_{n}},q_{d_{n}})$ at which the relaxation is tight for the optimal solution i.e. ${q}^{*}_{o_{l}} = \frac{ p^{*2}_{s(r)_{l}}+q^{*2}_{s(r)_{l}}}{V^{*}_{s(r)_{l}}} X_{l},\; \forall l \in \mathcal{L}$ or ${p}^{*}_{o_{\hat{l}}} = \frac{ p^{*2}_{s(r)_{\hat{l}}}+q^{*2}_{s(r)_{\hat{l}}}}{V^{*}_{s(r)_{\hat{l}}}} R_{\hat{l}},\; \forall l \in \mathcal{L}$.
\end{theorem}
\begin{proof}
To prove theorem \ref{theom:pd}, we firstly derive an equivalent model of the o-ACOPF model-\eqref{NOPF} by replacing constraint \eqref{eq:VsVr_chap2} with:
\begin{align}
V_{s_{l}}-v_{s_{l}}v_{r_{l}}cos \theta_l=p_{s_{l}}R_{l}+q_{s_{l}}X_{l},\; \forall l\in \mathcal{L}\label{eq:vsvrcos}
\end{align}
Constraints \eqref{eq:vsvrcos}-and-\eqref{eq:vsvrsin} are equivalent with constraints \eqref{eq:VsVr_chap2}-and-\eqref{eq:vsvrsin} for $(\theta^{min}_l,\theta^{max}_l) \subseteq (-\frac{\pi}{2}, \frac{\pi}{2})$ since they are expressing the same voltage drop phasor either in the way of real-and-imaginary parts or amplitude-and-imaginary parts illustrated in Fig. \ref{fig:v_drop_phasor}. 
\begin{figure}[!htbp]
\vspace{-0.5cm}
\hspace*{-0.9cm}
\centering
\includegraphics[width=10cm]{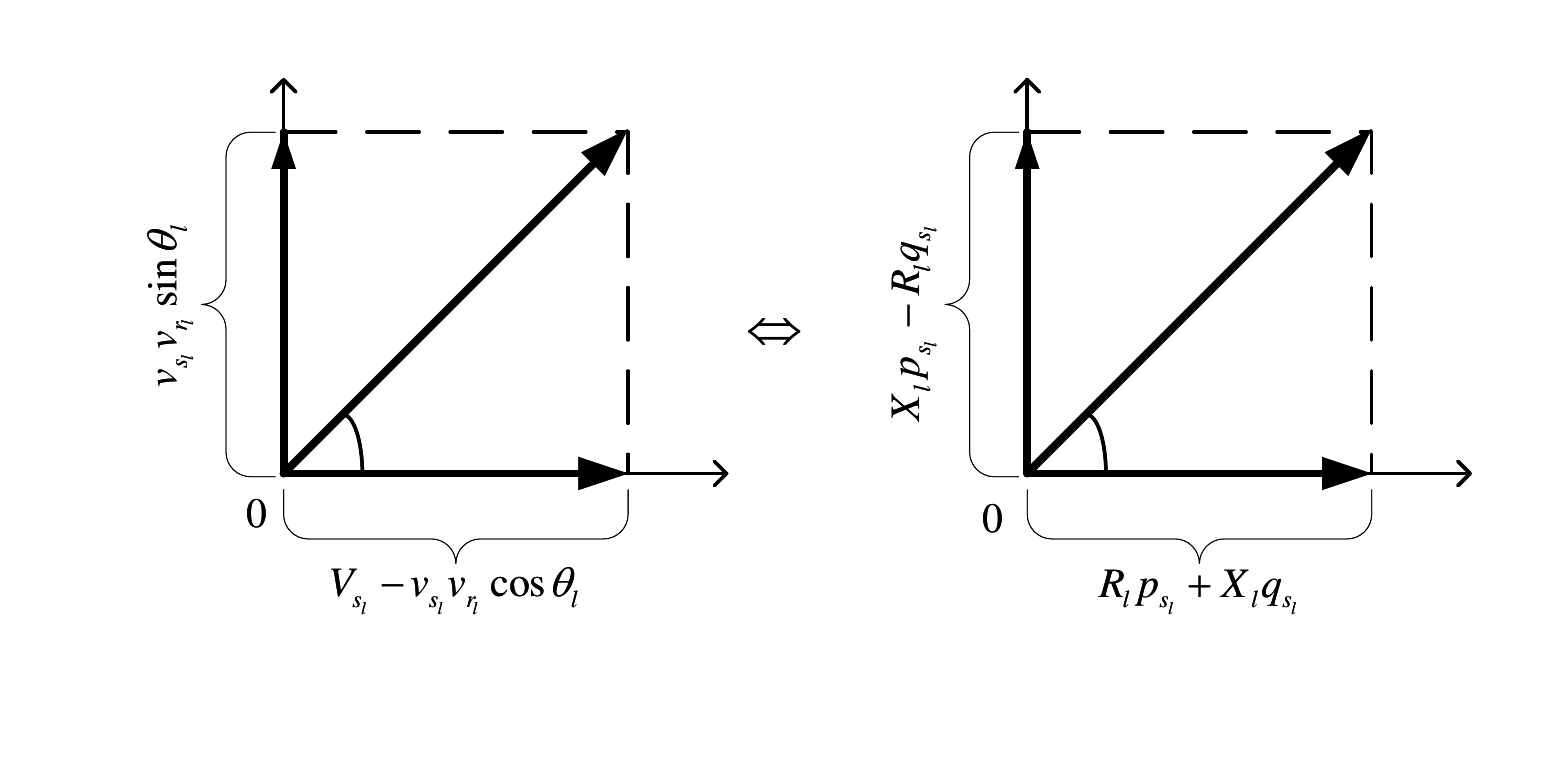}
\vspace{-1.5cm}
\caption{Equivalent Constraint of the Branch Voltage Drop Phasor}\label{fig:v_drop_phasor}
\end{figure}
We then define a new relaxed ACOPF model as r-ACOPF in \eqref{r-ACOPF}.
\begin{align}\label{r-ACOPF}
    &\underset{\Omega}{\text{Minimize}} {\hspace{10pt}}  f(\Omega)\\
    &\text{subject to} {\hspace{10pt}} \eqref{eq:p_balance_chap1}-\eqref{eq:q_balance_chap1}, \eqref{eq:vsvrsin},\eqref{eq:veq},\eqref{eq:vbound}-\eqref{eq:qbound} \nonumber\\
    &{\hspace{62pt}} \eqref{eq:plosscone}-\eqref{eq:pqloss}, \eqref{eq:vsvrcos}\nonumber
\end{align}
Obviously, when the relaxation is tight, the feasible region of the r-ACOPF model-\eqref{r-ACOPF} is actually equivalent to the feasible region of the SOC-ACOPF model-\eqref{SOPF-P} (with the additional constraint \eqref{eq:necerecover}) since any feasible solution from either model (SOC-ACOPF {model-}\eqref{SOPF-P} or r-ACOPF {model-}\eqref{r-ACOPF}) can be mapped to the feasible region of another model (SOC-ACOPF {model-}\eqref{SOPF-P} or r-ACOPF {model-}\eqref{r-ACOPF}) using the procedures we derived in theorems \ref{theom:relax}-\ref{theom:global}. The feasible region comparison is shown in Fig. \ref{fig:compare_region}.
\begin{figure}[!htbp]
\vspace{-0.7cm}
\centering
\includegraphics[width=9.5cm]{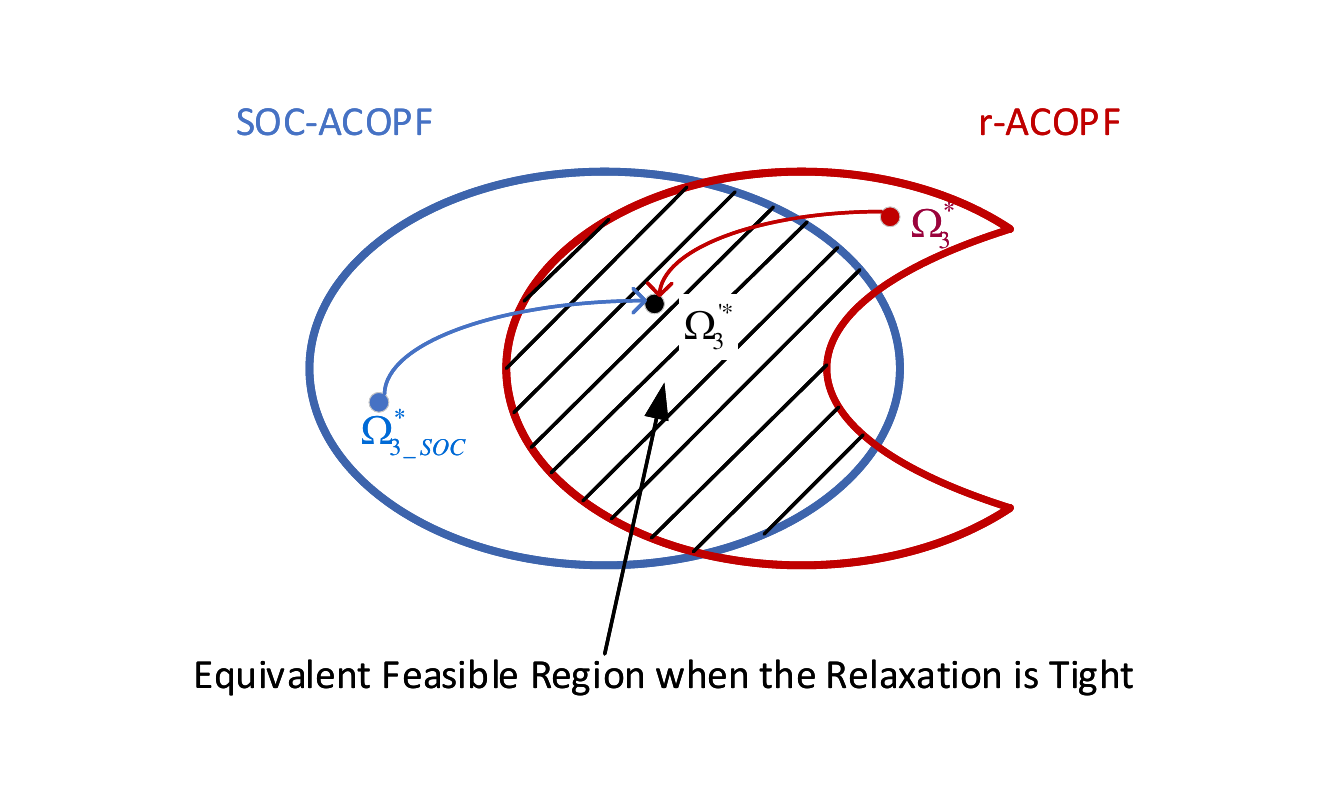}
\vspace{-1.2cm}
\caption{Comparison of the Feasible Regions of SOC-ACOPF and r-ACOPF}\label{fig:compare_region}
\end{figure}

We want to emphasize here that, although the r-ACOPF model-\eqref{r-ACOPF} is nonconvex (with regarding to its own set of variables), it is valid that part of its feasible region can be mapped to a convex one (with another set of variables) such as the convex SOC-ACOPF model-\eqref{SOPF-P} (with additional constraint \eqref{eq:necerecover}). An obvious simple example is the equivalence between the nonconvex region expressed by $\left\{ y=x^2,\, x>0 \right\}$ for $(y,x)\in \Re$ and convex region expressed by $\left \{ y=z,\, z>0 \right \}$ for $(y,z) \in \Re$. Fig. \ref{fig:equiva_region} shows this equivalence graphically. A similar consideration can be drawn for the r-ACOPF model-\eqref{r-ACOPF} (where both $V_n=v^{2}_n$ and $v_n$ are deployed) and the SOC-ACOPF model-\eqref{SOPF-P} where only $V_{n}$ are used.
\begin{figure}[!htbp]
\centering
\includegraphics[width=8cm]{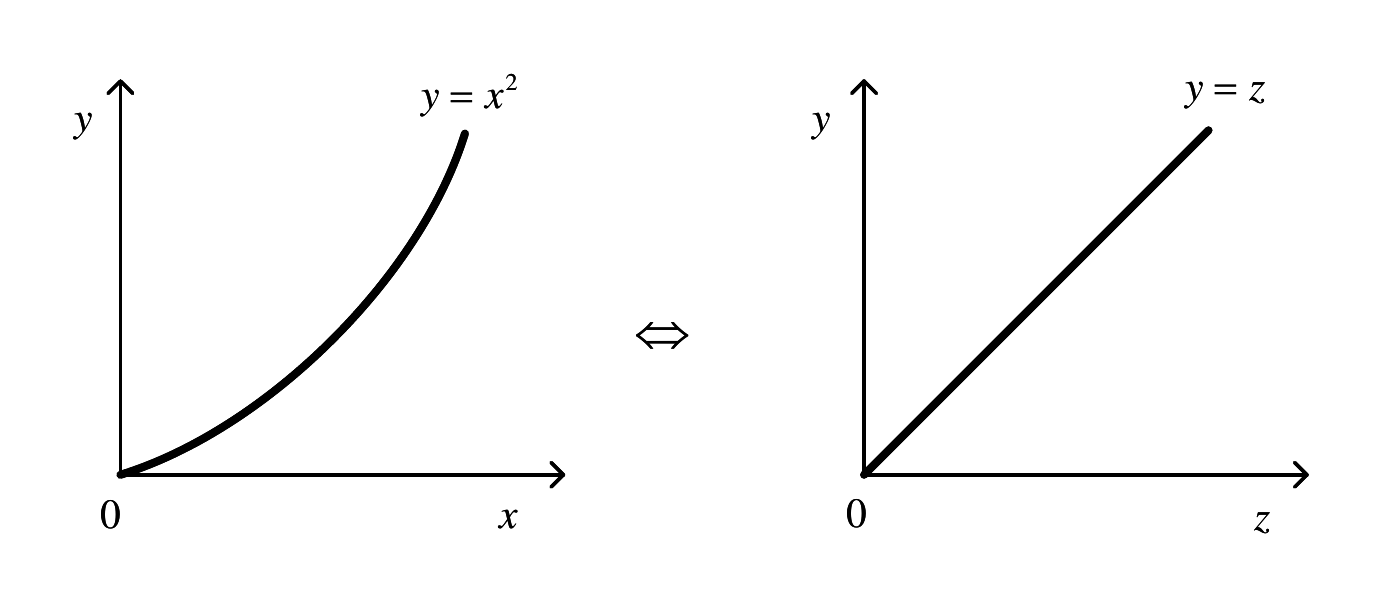}
\vspace{-0.6cm}
\caption{The Equivalence of Regions Expressed by Different Variables}\label{fig:equiva_region}
\end{figure}

Similarly, we define the (parametric) optimal value function of the r-ACOPF model-\eqref{r-ACOPF} (with regarding to $(p_{d_{n}},q_{d_{n}})$) as:
\begin{align}
  f^{*}(p_{d_{n}},q_{d_{n}}):=&{min}\; f(\Omega,p_{d_{n}},q_{d_{n}}) \nonumber\\
  &s.t.\;\; \left \{ \Omega \in \Omega_{r-ACOPF} \right \} 
\end{align}
Following the same procedure in the proof of theorem \ref{theom:mono}, it is obvious that $f^{*}$ here is also monotonic.

We now use the r-ACOPF model-\eqref{r-ACOPF} as a bridge to prove theorem \ref{theom:pd}. The key procedure is to show that theorem \ref{theom:pd} is valid for the r-ACOPF model-\eqref{r-ACOPF} and thus valid for the SOC-ACOPF model-\eqref{SOPF-P} considering the equivalence between their feasible regions (when the relaxation is tight). Suppose the optimal solution $\Omega^{*}_{3}$ of the r-ACOPF model-\eqref{r-ACOPF} has positive relaxation gap ${q}^{*}_{o_{3,\hat{l}}} > \frac{ p^{*2}_{3,s(r)_{\hat{l}}}+q^{*2}_{3,s(r)_{\hat{l}}}}{V^{*}_{3,s(r)_{\hat{l}}}} X_{\hat{l}}$ or ${p}^{*}_{o_{3,\hat{l}}} > \frac{ p^{*2}_{3,s(r)_{\hat{l}}}+q^{*2}_{3,s(r)_{\hat{l}}}}{V^{*}_{3,s(r)_{\hat{l}}}} R_{\hat{l}}$ for some $\hat{l} \in \mathcal{L}$ at $(p_{d_{n}},q_{d_{n}})$, we use $\Delta^{*}p_{o_{3,\hat{l}}}>0$ and $\Delta^{*}q_{o_{3,\hat{l}}}>0$ to represent the active and reactive power losses relaxation gaps respectively:
\begin{align}
&{p}^{*}_{o_{3,\hat{l}}} - \frac{ p^{*2}_{3,s(r)_{\hat{l}}}+q^{*2}_{3,s(r)_{\hat{l}}}}{V^{*}_{3,s(r)_{\hat{l}}}} R_{\hat{l}}=\Delta^{*}p_{o_{3,\hat{l}}}\label{plossgap} \\
&{q}^{*}_{o_{3,\hat{l}}} - \frac{ p^{*2}_{3,s(r)_{\hat{l}}}+q^{*2}_{3,s(r)_{\hat{l}}}}{V^{*}_{3,s(r)_{\hat{l}}}} X_{\hat{l}}=\Delta^{*}q_{o_{3,\hat{l}}}\label{qlossgap}
\end{align}
So the tight power losses solutions are ${p}^{'*}_{o_{3,\hat{l}}}={p}^{*}_{o_{3,\hat{l}}}-\Delta^{*}p_{o_{3,\hat{l}}}>0$ and ${q}^{'*}_{o_{3,\hat{l}}}={q}^{*}_{o_{3,\hat{l}}}-\Delta^{*}q_{o_{3,\hat{l}}}>0$. From constraints \eqref{eq:p_balance_chap1}-\eqref{eq:q_balance_chap1}, we have:
\begin{align}
&p^{*}_{3,n}-p_{d_{n}}=\sum_{l}({A^{+}_{nl}}p^{*}_{s_{3,l}}-{A^{-}_{nl}}p^{*}_{o_{3,l}})+G_{n}V^{*}_{3,n} \label{eq:p_balance6}
\end{align}
\begin{align}
&q^{*}_{3,n}-q_{d_{n}}=\sum_{l}({A^{+}_{nl}}q^{*}_{s_{3,l}}-{A^{-}_{nl}}q^{*}_{o_{3,l}})-B_{n}V^{*}_{3,n}\label{eq:q_balance6}
\end{align}
Suppose at $(p^{'}_{d_{n}},q^{'}_{d_{n}})$, tight power losses solutions are found (note in the r-ACOPF model-\eqref{r-ACOPF} only {the} following constraints are associated with the power losses variables aside from constraints \eqref{eq:plosscone}-\eqref{eq:pqloss}):
\begin{align}
&p^{*}_{3,n}-p^{'}_{d_{n}}=\sum_{l}({A^{+}_{nl}}p^{*}_{s_{3,l}}-{A^{-}_{nl}}p^{'*}_{o_{3,l}})+G_{n}V^{*}_{3,n} \label{eq:p_balance7}\\
&q^{*}_{3,n}-q^{'}_{d_{n}}=\sum_{l}({A^{+}_{nl}}q^{*}_{s_{3,l}}-{A^{-}_{nl}}q^{'*}_{o_{3,l}})-B_{n}V^{*}_{3,n}\label{eq:q_balance7}
\end{align}
Combining equations \eqref{eq:p_balance6}-\eqref{eq:q_balance6} with equations \eqref{eq:p_balance7}-\eqref{eq:q_balance7}, we have:
\begin{align}
&p^{'}_{d_{n}}=p_{d_{n}}-\sum_{l}{A^{-}_{nl}}\Delta^{*}p^{*}_{o_{3,l}} \label{eq:pdr}\\
&q^{'}_{d_{n}}=q_{d_{n}}-\sum_{l}{A^{-}_{nl}}\Delta^{*}q^{*}_{o_{3,l}} \label{eq:qdr}
\end{align}
Because ${A^{-}_{nl}}\in \left\{0,\, -1\right\}$ and $(\Delta^{*}p^{*}_{o_{3,l}},\Delta^{*}q^{*}_{o_{3,l}})>(0,0)$,  equations \eqref{eq:pdr}-\eqref{eq:qdr} imply $(p^{'}_{d_{n}},q^{'}_{d_{n}})>(p_{d_{n}},q_{d_{n}})$. Taking into account that the optimal value function $f^{*}$ is monotonic, we also know that $f^{*}(p_{d_{n}},q_{d_{n}})<f^{*}(p^{'}_{d_{n}},q^{'}_{d_{n}})$ which means $\Omega^{*}$ is not only feasible but also optimal for the r-ACOPF model-\eqref{r-ACOPF} at $(p^{'}_{d_{n}},q^{'}_{d_{n}})$. As we have explained at the beginning of this proof, since the r-ACOPF model-\eqref{r-ACOPF} is actually equivalent (in terms of the feasible region and optimal solution when the relaxation is tight) with the SOC-ACOPF model-\eqref{SOPF-P}, all the tight solutions we derived in proving theorem \ref{theom:pd} for the r-ACOPF model-\eqref{r-ACOPF} can always be mapped to the corresponding tight solution of the SOC-ACOPF model-\eqref{SOPF-P}. In other words, theorem \ref{theom:pd} is also valid for the SOC-ACOPF model-\eqref{SOPF-P}.
\end{proof}
Theorem \ref{theom:pd} shows the structure of the relaxed solutions and tight solutions for the SOC-ACOPF model-\eqref{SOPF-P} with regarding to the power load parameter. This theorem shows a counter-intuitive property of the SOC-ACOPF model-\eqref{SOPF-P} that larger power load can help tighten the relaxation gap. We present in the next Section the numerical validations of all the theorems and analysis in this Section.
\end{subequations}
\begin{table*}[!htbp]
  \centering
  \caption{Objective Solutions}
  \resizebox{\textwidth}{!}{%
    \begin{tabular}{|c|c|c|c|c|c|c|c|c|}\hline
    \multirow{2}[0]{*}{Test case} & \multicolumn{2}{c|}{10\% Load} & \multicolumn{2}{c|}{20\% Load} & \multicolumn{2}{c|}{30\% Load} & \multicolumn{2}{c|}{40\% Load} \\\cline{2-9}
          & SOC-ACOPF & o-ACOPF & SOC-ACOPF & o-ACOPF & SOC-ACOPF & o-ACOPF & SOC-ACOPF & o-ACOPF \\\hline
    {case4} &2.40107  &2.40107  &4.80428  &4.80428  &7.20965  &7.20965  &9.6171  &9.6171  \\\hline
    case9 &1170.74  &1170.75  &1347.23  &1347.23  &1593.64  &1593.64  &1909.78  &1909.78  \\\hline
    {13feeder} &0.00745  &0.00745  &0.01493  &0.01493  &0.02245  &0.02245  &0.03002  &0.03002  \\\hline
    IEEE14 & 545.64 & 546.37 & 1147.21 & 1147.50 & 1806.10 & 1806.26 & 2523.77 & 2523.91 \\\hline
    case30 & 33.14 & 33.14 & 75.31 & 75.31 & 123.61 & 123.61 & 178.12 & 178.12 \\\hline
    {34feeder} &0.00211  &0.00211  &0.00426  &0.00426 &0.00645  &0.00645  &0.00868 &0.00868  \\\hline
    IEEE57 & 2682.55 & 2686.41 & 5706.04 & 5709.00 & 9080.48 & 9082.93 & 12809.00 & 12810.65 \\\hline
    IEEE118 & 8940.49 & 8952.62 & 18735.71 & 18750.11 & 29420.72 & 29436.19 & 41008.27 & 41025.36 \\\hline
    ACTIVSg200 & 14070.44  & 14070.44  & 14070.44  & 14070.44  & 14070.44  & 14070.44  & 14483.45 & 14483.82 \\\hline
    IEEE300 & 51210.16 & 56915.23 & 107284.01 & 108378.18 & 168588.72 & 168712.09 & 235157.51 & 235244.97  \\\hline
    ACTIVSg500 & 16386.94  & 16386.94  & 16386.94 & 16386.94  & 16386.94  & 16386.94  & 16386.94  & 16386.94 \\\hline
    1354pegase & 7558.35 & 7558.47 & 15101.85 & 15102.06 & 22665.28 & 22665.88 & 30246.88 & 30249.40 \\\hline
    ACTIVSg2000 &301722.90   &301722.90   &319936.30   &321123.46 &401882.50   &402289.24   &509723.50  &510082.88  \\\hline
    2383wp     &0.00     &0.00     &0.00  &0.00     &19377.99 &19477.73    &175742.10     &175881.88 \\\hline
    2736sp     &0.00     &0.00     &0.00  &0.00     &108775.80   &108838.81      &250447.80   &250617.91 \\\hline
    2737sop     &0.00    &0.00     &5.86  &5.86    &42096.36    &42110.97      &128947.20   &128979.20 \\\hline
    2869pegase & 14639.05 & 14754.46 & 29204.10 & 29236.82 & 43811.78 & 43830.34 & 58458.84 & 58467.05\\\hline
    3012wp &0.00  &0.00   &0.00  &0.00  &41101.36  &41261.04  &338807.19  &339179.85  \\\hline
{3120sp} &1336861.76  &1336861.82  &1336861.76  &1336861.82  &1336861.76  &1336861.82  &1336861.76  &1336861.82\\\hline
    {3375wp} &6418725.22  &6418725.36  &6418725.22  &6418725.37  &6418725.22  &6418725.36  &6418725.22  &6418725.36  \\\hline

    \end{tabular}}
  \label{tab:objective}%
\end{table*}%
\begin{table*}[!htbp]
  \centering
  \caption{Maximum Relaxation Gaps of Active Power Losses from the SOC-ACOPF Model}
  \resizebox{\textwidth}{!}{%
    \begin{tabular}{|c|c|c|c|c|c|c|c|c|}\hline
    Test case &{10\% Load} &{$>$10\% Load} & {20\% Load} &{$>$20\% Load} & {30\% Load} &{$>$30\% Load} &{40\% Load} &{$>$40\% Load} \\\hline
    {case4} &9.75E-11  &8.46E-10  &1.00E-10  &8.27E-10  &1.00E-10  &6.82E-10  &9.85E-11  &8.51E-10  \\\hline
    case9 &9.82E-11 & 1.59E-16 & 5.65E-11 & 3.60E-13 & 1.16E-10 & 2.95E-13 & 9.81E-11 & 3.02E-12 \\\hline
    {13feeder} &1.24E-7  &6.42E-9  &1.24E-7  &6.43E-9  &1.23E-7  &6.38E-9  &1.22E-7  &6.19E-9  \\\hline
    IEEE14 & 5.43E-10 & 7.31E-13 & 3.35E-09 & 6.89E-13 & 1.06E-10 & 6.85E-12 & 2.99E-09 & 1.45E-12  \\\hline
    case30 & 5.45E-10 & 6.47E-10 & 1.67E-10 & 1.33E-11 & 1.71E-09 & 1.72E-10 & 1.21E-10 & 1.84E-10 \\\hline
    {34feeder} &1.25E-7  &6.98E-9  &1.25E-7  &6.82E-9  &1.25E-7  &6.68E-9  &1.25E-7  &2.40E-9  \\\hline
    IEEE57 & 1.84E-10 & 5.92E-10 & 2.52E-11 & 7.46E-10 & 1.36E-09 & 2.09E-10 & 6.38E-09 & 8.26E-10  \\\hline
    IEEE118 & 4.93E-09 & 3.28E-09 & 6.02E-08 & 1.79E-09 & 1.11E-07 & 1.10E-09 & 1.57E-09 & 3.99E-11  \\\hline
    ACTIVSg200 & 1.67E-02 & 1.28E-10 & 1.43E-02 & 3.37E-10 & 6.65E-03 & 6.97E-10 & 1.45E-08 & 5.79E-10 \\\hline
    IEEE300 & 1.15E-02 & 4.14E-08 & 1.37E-03 & 8.17E-09 & 4.08E-06 & 2.52E-09 & 5.01E-05 & 1.10E-09  \\\hline
    ACTIVSg500 & 2.29E-02 & 1.34E-09 & 2.12E-02 & 3.69E-09 & 1.93E-02 & 5.49E-10 & 1.60E-02 & 3.59E-10 \\\hline
    1354pegase & 3.09E-03 & 1.79E-07 & 2.72E-07 & 1.42E-07 & 1.01E-02 & 1.43E-07 & 1.04E-02 & 1.03E-07 \\\hline
    ACTIVSg2000 &2.14E-01   &5.55E-08   &1.09E-01   &9.87E-09   &5.26E-02   &1.93E-08   &4.10E-02  &1.03E-08  \\\hline
    2383wp     &1.94E-01  &3.65E-08  &1.77E-01  &3.11E-08  &2.52E-03  &1.28E-08  &7.46E-03  &4.29E-08 \\\hline
    2736sp     &2.10E-02  &9.68E-09  &3.22E-10  &2.57E-08  &2.85E-08  &9.64E-09  &3.58E-04  &5.74E-10 \\\hline
    2737sop     &3.32E-02 &5.12E-08  &1.45E-02  &3.75E-08  &4.56E-08  &2.89E-08  &1.65E-03  &2.35E-08 \\\hline
    2869pegase & 8.47E-03 & 7.12E-06 & 9.32E-03 & 1.38E-07 & 6.72E-03 & 6.56E-08 & 1.47E-02 & 6.99E-08  \\\hline
    3012wp & 1.09E-01 & 4.21E-06 & 2.35E-02 & 4.64E-09 & 1.69E-03 & 3.78E-11 & 4.75E-04 & 3.08E-11 \\\hline
    {3120sp} &1.01E+01  &2.85E-09  &7.48E+00  &1.72E-11  &3.56E+00  &2.72E-11  &1.28E+00  &5.18E-11  \\\hline
    {3375wp} &7.22E+00  &5.06E-10  &5.40E+00  &6.82E-8  &3.12E+00  &2.43E-10  &1.39E+00  &1.05E-8  \\\hline
    \end{tabular}}%
  \label{tab:gap_activeloss}%
\end{table*}%
\begin{table*}[!htbp]
  \centering
  \caption{Maximum Relaxation Gaps of Reactive Power Losses from the SOC-ACOPF Model}
  \resizebox{\textwidth}{!}{%
    \begin{tabular}{|c|c|c|c|c|c|c|c|c|c|c|}\hline
    Test case &{10\% Load} &{$>$10\% Load} & {20\% Load} &{$>$20\% Load} & {30\% Load} &{$>$30\% Load} &{40\% Load} &{$>$40\% Load} \\\hline
    {case4} &1.95e-10  &1.69e-9  &2.01e-10  &1.65e-9  &2.01e-10  &1.36e-9  &1.97e-10  &1.70e-9  \\\hline
    case9 &3.65E-01 & 1.16E-10 & 1.57E-01 & 3.06E-12 & 1.33E-01 & 4.70E-12 & 1.14E-01 & 1.63E-11 \\\hline
    {13feeder} &3.42E-7  &2.42E-9  &3.40E-7  &2.40e-9  &3.37E-7  &2.37E-9  &3.35E-7  &2.34E-9  \\\hline
    IEEE14 & 2.50E-01 & 8.51E-11 & 7.95E-02 & 3.92E-11 & 1.43E-01 & 1.24E-10 & 1.18E-01 & 2.33E-11 \\\hline
    case30 & 2.31E-02 & 1.78E-09 & 2.34E-02 & 2.65E-11 & 2.21E-02 & 3.66E-10 & 2.04E-02 & 3.79E-10 \\\hline
    {34feeder} &3.10E-7  &2.43E-9  &3.03E-7  &2.43E-9  &2.97E-7  &2.44E-9  &2.90E-7  &8.94E-10  \\\hline
    IEEE57 & 2.25E-01 & 2.00E-09 & 1.92E-01 & 9.65E-09 & 1.46E-01 & 3.10E-09 & 1.02E-01 & 3.37E-09  \\\hline
    IEEE118 & 2.98E+00 & 1.00E-08 & 3.08E+00 & 1.37E-08 & 2.86E+00 & 9.55E-09 & 2.84E+00 & 3.03E-09 \\\hline
    ACTIVSg200 & 2.05E-01 & 1.35E-09 & 1.99E-01 & 1.94E-09 & 2.09E-01 & 7.14E-09 & 1.06E-06 & 2.91E-09 \\\hline
    IEEE300 & 5.25E+00 & 9.33E-08 & 4.12E+00 & 1.75E-08 & 3.82E+00 & 3.26E-08 & 3.62E+00 & 6.55E-08  \\\hline
    ACTIVSg500 & 3.31E-01 & 1.41E-08 & 2.87E-01 & 4.88E-08 & 2.45E-01 & 7.27E-09 & 2.33E-01 & 4.74E-09 \\\hline
    1354pegase & 2.95E-01 & 1.46E-06 & 2.93E-01 & 3.45E-07 & 8.50E-01 & 4.77E-07 & 8.74E-01 & 2.63E-07  \\\hline
    ACTIVSg2000 &2.23E+00   &2.10E-07   &2.57E+00   &8.57E-08   &2.56E+00   &1.36E-07   &2.53E+00  &4.70E-08  \\\hline
    2383wp     &1.87E+00  &2.85E-07  &1.61E+00  &2.34E-07  &1.13E-01  &4.60E-08  &3.34E-01  &8.70E-08 \\\hline
    2736sp     &2.37E-01  &4.94E-08  &1.19E-03  &1.21E-07  &1.18E-03  &3.93E-08  &1.48E-02  &1.74E-09 \\\hline
    2737sop     &3.76E-01 &1.82E-07  &1.58E-07  &3.13E-07  &1.19E-03  &1.76E-07  &1.17E-01  &1.39E-06 \\\hline
    2869pegase & 7.81E-01 & 4.96E-05 & 5.03E-01 & 3.44E-06 & 3.79E-01 & 8.79E-07 & 2.64E+00 & 1.12E-06  \\\hline
    3012wp & 6.23E-01 & 1.01E-10 & 3.36E-01 & 1.00E-10 & 1.54E+00 & 1.16E-10 & 1.33E+00 & 9.91E-11 \\\hline
    {3120sp} &3.10E+01  &1.01E-10  &1.29E+01  &9.92E-11  &6.13E+00  &9.90E-11  &1.58E+00 &9.88e-11  \\\hline
    {3375wp} &7.52E+00  &4.72E-10  &4.67E+01  &4.67E-10  &3.12E+01  &4.56E-10  &2.75E+01  &4.45E-10  \\\hline
    \end{tabular}}%
  \label{tab:gap_reactiveloss}%
\end{table*}%
\section{Numerical results and discussion}
The SOC-ACOPF model-\eqref{SOPF-P} (with the additional constraint \eqref{eq:necerecover}) and o-ACOPF model-\eqref{NOPF} are implemented and solved in the YALMIP \cite{yalmip_toolbox} toolbox for modelling and optimization in MATLAB and Julia \cite{julia_fresh} running on the 64-bit macOS operating system. A personal computer with Intel Core i9 2.9 GHz CPU and 16G RAM is deployed. The MOSEK, Gurobi, SCS solvers are used to solve the SOC-ACOPF model-\eqref{SOPF-P}. Because MOSEK, Gurobi and SCS can only solve convex models, the convexity of the SOC-ACOPF model-\eqref{SOPF-P} is numerically validated. The IPOPT solver is used to solve the nonconvex o-ACOPF model-\eqref{NOPF}.
\subsection{{Validation of Theorems} \ref{theom:relax}-\ref{theom:pd}}
We use {benchmark} IEEE {transmission} networks \{14, 57, 118, 300\}, matpower case {9, 30}, ACTIVSg \{200, 500, 2000\}, 2383wp, 2736sp, 2737sop, 3012wp, {3120sp, 3375wp,} \{1354, 2869\} pegase \cite{Fliscounakis} {and the distribution networks case4, 13feeder, 34feeder} as the test cases to {numerically} validate {the} proposed theorems {for both power transmission and distribution networks} in this Section. Power network data from MATPOWER are directly used here \cite{Zimmerman2011}. {Note the 3012wp, 3120sp and 3375wp networks are stressed test cases for the winter peak (wp) time or summer peak (sp) time.} The objective function $f$ of typical economic dispatch is quadratic i.e. $f(p_{{n}})=\sum_{n}(\alpha_n p^2_{{n}}+\beta_n p_{{n}}+\gamma_{n})$. Where $(\alpha_{n}, \beta_{n}, \gamma_{n}) \geq 0$ are the cost parameters of nodal active power generation. Note some cost parameters {of the generators} are equal to zero in the test cases. So in some low power load scenarios, the total generation cost is equal to zero {or remains the same with the increase of power loads}. {For the distribution networks, we include several distributed generators with positive cost parameters in all the networks.}
To make sure the objective function $f$ is monotonic for $p_{{n}}$, we must ensure $p^{min}_{{n}} \geq 0$ in the o-ACOPF and SOC-ACOPF models. For large test cases with over 3000 nodes, we use Julia to implement and solve both the SOC-ACOPF model-\eqref{SOPF-P} and the o-ACOPF model-\eqref{NOPF}. {This is mainly due to the higher computational efficiency of Julia compared with YALMIP with the associated optimization solvers.}

Since we have tested our models for the base power loads (original power loads in MATPOWER) and shown in our previous work \cite{SOCACOPF_relax_fea_me} that large relaxation gaps are present for optimal solutions in low power loads scenarios, we only show the results for low power loads scenarios here. In Table \ref{tab:objective}, we list the objective solutions of the SOC-ACOPF model-\eqref{SOPF-P} and the o-ACOPF model-\eqref{NOPF} for 10\%, 20\%, 30\% and 40\% of the absolute value of base load. So, the results in Table \ref{tab:objective} validate theorems 1-5:
\begin{itemize}
\item The SOC-ACOPF model-\eqref{SOPF-P} (with the additional constraint \eqref{eq:necerecover}) is a relaxation of the o-ACOPF model-\eqref{NOPF}. This is validated since all the objective solutions of the SOC-ACOPF model-\eqref{SOPF-P} are lower than the {corresponding objective solutions of the} o-ACOPF model-\eqref{NOPF}.

\item The (parametric) optimal value functions of the SOC-ACOPF model-\eqref{SOPF-P} and the o-ACOPF model-\eqref{NOPF} are monotonic (with regarding to the power load $p_{d}$) given that the objective function $f$ is monotonic with regarding to the power generation $p_{n}$. This is validated since all the objective solutions increase when the power loads increase.
\end{itemize}
For the SOC-ACOPF model-\eqref{SOPF-P}. The maximum relaxation gaps of active and reactive power losses are defined as:
\begin{subequations}
\begin{align}
&Gap^{po,max}:= Maximum \left \{Gap^{po}_{l},\, \forall l \in \mathcal{L} \right\}\\
&Gap^{qo,max}:= Maximum \left\{ Gap^{qo}_{l},\, \forall l \in \mathcal{L} \right\}
\end{align}
\end{subequations}
We use $Gap^{po,max}$ and $Gap^{qo,max}$ to measure {the} relaxation performance of the SOC-ACOPF model-\eqref{SOPF-P} (smaller relaxation gaps mean better performance). We list the maximum relaxation gaps results $Gap^{po,max}$ and $Gap^{qo,max}$ of the SOC-ACOPF model-\eqref{SOPF-P} in Table \ref{tab:gap_activeloss} and Table \ref{tab:gap_reactiveloss}. When the relaxation is not tight, we increase the power load levels $(p_{d},q_{d})$ (but fix the optimal nodal power generation solution $p^{*}_{n}$ so the objective solutions remain the same as in Table \ref{tab:objective}) to find the tight solutions of the SOC-ACOPF model-\eqref{SOPF-P}. The power load levels are increased by taking the power loads $p_{d}, q_{d}$ as variables in the SOC-ACOPF model and using the original power load levels (10\%, 20\%, 30\% and 40\%) as the lower bounds of these variables. The maximum relaxation gaps for tight solutions are listed in the columns denoted as '$>$10\% Load', '$>$ 20\% Load', '$>$ 30\% Load' and '$>$ 40\% Load' of Table \ref{tab:gap_activeloss} and Table \ref{tab:gap_reactiveloss}. So we have actually numerically validated theorem \ref{theom:pd} by the results in Table \ref{tab:gap_activeloss} and Table \ref{tab:gap_reactiveloss}. It is worth to mention that for most branches, the relaxation gaps are very small. The maximum relaxation gap only appear for single branch in each test case i.e. there is only one branch $\tilde{l} \in \mathcal{L}$ in each test case such that $Gap^{po}_{\tilde{l}}=Gap^{po,max}$ or $Gap^{qo}_{\tilde{l}}=Gap^{qo,max}$.
\begin{table*}[!htbp]
  \centering
  \caption{{Objective Solution}}
    \begin{tabular}{|c|c|c|c|c|}\hline
    Test case & SOC-ACOPF & o-ACOPF &OPF Model in \cite{qc_relax_minlp} &OPF Model in \cite{Jabr2006} \\\hline
    case4 & 24.11 & 24.11 & 24.11 & 24.11 \\\hline
    case9 & 5296.69 & 5296.69 & 5296.67 & 5296.67 \\\hline
    13feeder & 0.0954 & 0.0954 & 0.0954 & 0.0954 \\\hline
    IEEE14 & 8081.55 & 8081.61 & 8075.12 & 8075.12 \\\hline
    case30 & 576.85 & 576.89 & 573.58 & 573.58 \\\hline
    34feeder & 0.0231 & 0.0231 & 0.0231 & 0.0231 \\\hline
    IEEE57 & 41735.91 & 41738.11 & 41711.01 & 41711.00 \\\hline
    IEEE118 & 129626.18 & 129660.63 & 129341.96 & 129341.94 \\\hline
    ACTIVsg200 & 27557.57 & 27557.57 & 27556.65 & 27556.64 \\\hline
    IEEE300 & 719699.91 & 719732.11 & 718686.06 & 718654.17 \\\hline
    ACTIVsg500 & 71892.37 & 71816.43 & 68673.23 & 68672.61 \\\hline
    1354pegase & 74060.13 & 74068.93 & 74015.05 & 74012.27 \\\hline
    ACTIVsg2000 & 1228772.15 & 1228892.07 & 1226346.27 & 1226328.88 \\\hline
    2383wp & 1857584.78 & 1858356.48 & 1849952.56 & 1848900.95 \\\hline
    2736sp & 1307281.14 & 1307708.23 & 1303994.89 & 1303975.55 \\\hline
    2737sop & 777545.36 & 777863.88 & 775693.24 & 775673.19 \\\hline
    2869pegase & 133990.51 & 134000.45 & 133881.91 & 133879.80 \\\hline
    3012wp & 2580154.20 & 2582184.55 & 2572123.51 & 2571508.81 \\\hline
    3120sp & 2137388.24 & 2138775.34 & 2131420.72 & 2131316.09 \\\hline
    3375wp & 7402736.38 & 7404278.88 & NA    & 7393007.21 \\\hline
    \end{tabular}%
  \label{tab:comp_obj}%
\end{table*}%
\begin{table*}[!htbp]
  \centering
  \caption{{CPU Time}}
    \begin{tabular}{|c|c|c|c|c|}\hline
    Test case & SOC-ACOPF & o-ACOPF &OPF Model in \cite{qc_relax_minlp} &OPF Model in \cite{Jabr2006} \\\hline
    case4 & 0.01  & 0.01  & 0.04  & 0.04 \\\hline
    case9 & 0.02  & 0.02  & 0.04  & 0.02 \\\hline
    13feeder & 0.02  & 0.02  & 0.06  & 0.02 \\\hline
    IEEE14 & 0.06  & 0.02  & 0.08  & 0.04 \\\hline
    case30 & 0.09  & 0.10  & 0.15  & 0.05 \\\hline
    34feeder & 0.07  & 0.03  & 0.17  & 0.07 \\\hline
    IEEE57 & 0.13  & 0.14  & 0.32  & 0.09 \\\hline
    IEEE118 & 0.40  & 0.37  & 0.92  & 0.21 \\\hline
    ACTIVsg200 & 0.60  & 0.52  & 1.19  & 0.35 \\\hline
    IEEE300 & 1.20  & 0.79  & 2.98  & 0.87 \\\hline
    ACTIVsg500 & 2.59  & 2.00  & 4.49  & 1.52 \\\hline
    1354pegase & 5.47  & 4.27  & 35.74 & 19.00 \\\hline
    ACTIVsg2000 & 86.48 & 28.72 & 40.34 & 11.43 \\\hline
    2383wp & 52.62 & 37.26 & 62.51 & 20.88 \\\hline
    2736sp & 31.76 & 28.58 & 57.67 & 16.19 \\\hline
    2737sop & 21.54 & 16.73 & 60.13 & 12.12 \\\hline
    2869pegase & 19.66 & 11.94 & 83.10 & 70.72 \\\hline
    3012wp & 42.05 & 40.09 & 63.35 & 17.83 \\\hline
    3120sp & 60.62 & 943.11 & 85.69 & 21.98 \\\hline
    3375wp & 89.19 & 38.82 & NA    & 21.40 \\\hline
   \end{tabular}%
  \label{tab:comp_time}%
\end{table*}%
\subsection{{Comparison with Other OPF Models}}
{To compare the accuracy and computational efficiency of our o-ACOPF model-}\eqref{NOPF} 
{and SOC-ACOPF model-}\eqref{SOPF-P} {with other convex OPF models, we run the numerical test for all the test cases with the original power loads. The other two convex OPF models are based on QC relaxation and SOC relaxation [11,16]. All the models and numerical tests are implemented in Julia. The optimal objective results are listed in Table} \ref{tab:comp_obj}. {The computational CPU time results are listed in Table} \ref{tab:comp_time}. 
{If the model cannot converge, we use 'NA' to denote the result. Our models converge for all the test cases. The objective solutions are very close to the other two convex OPF models. For the 3375wp test case, since it is a stressed test case representing the 2007-08 evening peak in Poland, the convex OPF model in}  [11] {cannot converge. For the computation CPU time, our SOC-ACOPF model-}\eqref{SOPF-P} {requires less computational time than the convex OPF model in} [11] {but longer time than the convex OPF model in} [16]. {Our o-ACOPF model-}\eqref{NOPF} {takes the least computational time among all the considered models. It is worth to mention that only our o-ACOPF model-}\eqref{NOPF} {and SOC-ACOPF model-}\eqref{SOPF-P} {consider the difference between the actual measurable power flow and the non-measurable power flow as explained in Section II of this paper.} 
\section{Conclusions}
By equivalently deriving the branch ampacity constraint for the power losses, we firstly remove the approximation gap of the SOC-ACOPF model-\eqref{SOPF-P} based on the branch flow formulation. We then give the analytical and numerical proofs for several important properties of the SOC-ACOPF model-\eqref{SOPF-P}. The aims of proving these properties are to improve the solution quality and to promote the applicability of the SOC-ACOPF model-\eqref{SOPF-P}. We show rigorously that the feasible region of the SOC-ACOPF model-\eqref{SOPF-P} covers the feasible region of the o-ACOPF model-\eqref{NOPF}. Regarding to the AC feasibility of the solution from the SOC-ACOPF model-\eqref{SOPF-P}, one additional necessary conic constraint to improve the AC feasibility of the solution is derived in theorem \ref{theom:recover}. We also show how to recover the global optimal solution of the o-ACOPF model-\eqref{NOPF} from the tight optimal solution of the SOC-ACOPF model-\eqref{SOPF-P} in theorem \ref{theom:global}. Based on the monotonic property of the optimal value functions defined by the SOC-ACOPF model-\eqref{SOPF-P} and the o-ACOPF model \eqref{NOPF} from theorem \ref{theom:mono}, we prove that the tight solutions of the SOC-ACOPF model-\eqref{SOPF-P} can always be obtained by allowing the increase of nodal power loads in theorem \ref{theom:pd}. This theorem shows that large power load levels actually help tighten the relaxation gap of the SOC-ACOPF model-\eqref{SOPF-P}. In other words, SOC-ACOPF model-(2) is more applicable for high power load conditions. {According to our numerical comparison with the other two convex OPF models in the literaure, our SOC-ACOPF model-(2) has similar accuracy and computational efficiency. However, only our SOC-ACOPF model-(2) derives the correct branch ampacity constraint by considering the difference between the actual measurable power flows and the nonmeasurable power flows in the transmission line $\Pi$-model. Furthermore, our o-ACOPF model-(1) has the best computational efficiency. Future research work is expected to consider more details about the power load models (such as the polynomial (ZIP) model or the exponential model) and discrete transformer taps as decision variables.}

\bibliographystyle{elsarticle-num}
\bibliography{OPF, pub_me, julia}
\end{document}